\documentclass[11pt, letter]{amsart}
\usepackage[euler-digits]{eulervm}

\usepackage[headings]{fullpage}

\usepackage{amsfonts, amssymb, amsmath, stmaryrd}
\usepackage{mathrsfs,array}
\usepackage{eucal,times,color,enumerate,accents}
\usepackage{url}
\usepackage{bbm}
\usepackage{microtype}
\usepackage{color}
\usepackage{tikz}
\usepackage{tikz-cd}
\usepackage{bm}
\usepackage{enumerate}
\usetikzlibrary{calc}
\usepackage{subfigure}

\renewcommand{\familydefault}{ppl}
\address{Department of Mathematics, Colorado State University}
\email{renzo@math.colostate.edu}

\address{School of Mathematics and Statistics, The University of Sheffield}
\email{paul.johnson@sheffield.ac.uk}

\address{Eberhard Karls Universit\"at T\"ubingen, Fachbereich Mathematik}
\email{hannah@math.uni-tuebingen.de}

\address{Department of Pure Mathematics and Mathematical Statistics, University of Cambridge}
\email{dr508@cam.ac.uk}

\newtheorem{theorem}{Theorem}[section]
\newtheorem{corollary}[theorem]{Corollary}
\newtheorem{lemma}[theorem]{Lemma}
\newtheorem{proposition}[theorem]{Proposition}
\newtheorem{definition}[theorem]{Definition}
\newtheorem{warning}[theorem]{Warning}

\newtheorem{construction}[theorem]{Construction}
\newtheorem{convention}[theorem]{Convention}
\newtheorem{quasi-theorem}[theorem]{Quasi-Theorem}

\newtheorem{rem1}[theorem]{Remark}
\newenvironment{remark}{\begin{rem1}\em}{\end{rem1}}

\newtheorem{ex1}[theorem]{Example}
\newenvironment{example}{\begin{ex1}\em}{\end{ex1}}

\newtheorem{not1}[theorem]{Notation}
\newenvironment{notation}{\begin{not1}\em}{\end{not1}}

\newcommand{\CC} {{\mathbb C}}          
            
\newcommand{\NN} {{\mathbb N}}		
\newcommand{\PP}{\mathbb{P}}         
		
\newcommand{\RR} {{\mathbb R}}		
\newcommand{\ZZ} {{\mathbb Z}}

\newcommand{\trop}{t\!r\!o\!p}

\newcommand{\floor}{f\!l\!o\!o\!r}

\usepackage{hyperref}
\hypersetup{
  colorlinks   = true,          
  urlcolor     = blue,          
  linkcolor    = violet,          
  citecolor   = blue             
}

\DeclareMathOperator{\Aut}{Aut}

\DeclareMathOperator{\val}{val}
\DeclareMathOperator{\mult}{\mathsf{mult}}

\DeclareMathOperator{\spec}{Spec}

\newcommand{\renzo}[1]{}
\newcommand{\rc}[1]{{  #1}}
\newcommand{\hannah}[1]{}
\newcommand{\dhruv}[1]{}
\newcommand{\paul}[1]{}

\title[Counting curves on surfaces: Tropical geometry \& the Fock space]{{\Large C}ounting curves on Hirzebruch surfaces \\ tropical geometry {\it \&} the Fock space}
\author[Cavalieri--Johnson--Markwig--Ranganathan]{Renzo Cavalieri \ \  Paul Johnson \ \  Hannah Markwig \ \  Dhruv Ranganathan}

\begin{document}

\maketitle

\begin{abstract}
We study the stationary descendant Gromov--Witten theory of toric surfaces by combining and extending a range of techniques -- tropical curves, floor diagrams, and Fock spaces. A correspondence theorem is established between tropical curves and descendant invariants on toric surfaces using maximal toric degenerations. An intermediate degeneration is then shown to give rise to floor diagrams, {giving} a geometric interpretation of this well-known bookkeeping tool in tropical geometry. In the process, we extend floor diagram techniques to include descendants in arbitrary genus. These floor diagrams are then used to connect tropical curve counting to the algebra of operators on the bosonic Fock space, and are shown to coincide with the Feynman diagrams of  appropriate operators. This extends work of a number of researchers, including Block--G\"ottsche, Cooper--Pandharipande, and Block--Gathmann--Markwig. 
\end{abstract}

\section{Introduction}

\subsection{Overview}

The scope of this manuscript is to explore the relationships between the following enumerative and combinatorial geometric theories of surfaces, studied by a number of researchers in the last decade:
\begin{enumerate}
	\item decorated floor diagram counting;
	\item logarithmic and relative Gromov--Witten theory of Hirzebruch surfaces;
	\item tropical descendant Gromov--Witten theory of Hirzebruch surfaces;
	\item matrix elements of operators on a bosonic Fock space.
	\end{enumerate}	

{\bf Floor diagrams} are loop free graphs on a linearly ordered set of vertices, further endowed with vertex, edge, and half-edge decorations as specified in Definition \ref{def-floor}. Each floor diagram is counted with a weight, coming from context in which it arises. Floor diagrams capture the combinatorial essence of the other three theories, in the sense that the simplest way to exhibit the above equivalences is through a weight preserving bijection between floor diagrams and specific ways to organize the enumeration in the other theories.

{\bf Gromov--Witten theory} studies the intersection theory on moduli spaces of maps from pointed curves to a target surface. We are concerned with two distinct flavours of this theory -- the relative and logarithmic invariants -- which impose tangency conditions along certain boundary divisors, as in Definition~\ref{def-logdescGWI} and Definition~\ref{def-reldescGWI}. These moduli spaces admit a virtual fundamental class, and zero dimensional cycles are constructed by capping with the virtual class two types of cycles: point conditions, corresponding to requiring a point on the curve to map to a specified point on the surface; and descendant insertions, which are Euler classes of certain tautological line bundles on the moduli space, associated to each marked point. The word \textit{stationary} refers to the fact that descendant insertions are always coupled with point conditions.
 In this work, we specify special tangency orders to the $0$ and $\infty$ sections of Hirzebruch surfaces, taking inspiration from the geometry of double Hurwitz numbers. In the logarithmic case, we specify \textit{transverse} contact along the torus invariant fibers. By using a degeneration of the relative geometry to a chain of Hirzebruch surfaces, in Theorem \ref{thm-flooralg}, the equivalence of the relative invariants with floor diagram counts is established. The relationship to logarithmic invariants is more subtle and passes through the tropical equivalence described below.
 
{\bf Tropical Gromov--Witten theory} of surfaces consists of the study of piecewise linear, balanced maps from tropical curves into $\RR^2$, see Definition \ref{def-tropdescGWI}.  One obtains a finite count by imposing point conditions (i.e. specifying the image of a contracted marked end on the plane), and tropical descendant conditions. The descendant conditions constrain the valency of the vertex  adjacent to a marked end. Each map is counted with a weight that arises as an intersection number on a certain moduli space of logarithmic stable maps. In good cases, these \rc{weights} can be further spread out as products \rc{of combinatorial factors} over the vertices. The directions and multiplicities of the infinite ends  define a Newton fan, which determines at the same time a toric surface, a curve class on it, and prescribed tangencies along the toric divisors, offering a natural candidate for a correspondence between the logarithmic and tropical theories.

The logarithmic theory is shown to coincide with the tropical count in Theorem \ref{thm-corres}, using the recently established \textit{decomposition formula} for logarithmic Gromov--Witten invariants~\cite{ACGS}.  The correspondence between the tropical count and the floor diagram count is established by a combinatorial argument. Specifically, after specializing the tropical point conditions, the contributing curves take a very special form, and become {\it floor decomposed}, meaning that  certain subgraphs of the tropical curves may be contracted to give rise to a floor diagram. The floor decomposition yields a nontrivial result for the logarithmic invariants -- namely, that the multiplicity of a floor decomposed tropical curve can be obtained in terms of the multiplicities associated to its vertices. A general such statement for logarithmic invariants is unknown, even for toric surfaces. An exploration of when such a vertex-local expression exists can be found in recent work of Mandel--Ruddat, where it is packaged as a ``tropical'' quantum field theory~\cite{MR19}. 

The {\bf bosonic Fock space} is a countably infinite dimensional vector space with a basis indexed by ordered pairs of partitions \rc{of positive integers}. It has an action of a Heisenberg algebra of operators, generated by two families of operators $a_s,b_s$ parameterized by the integers. The distinguished basis vectors can naturally be identified with tangency conditions along the $0$ and $\infty$ sections of a Hirzebruch surface. In Definition \ref{def-operator}, we construct a family of linear operators $M_{l}$ on the Fock space which are naturally associated to stationary descendant insertions.
To each (relative or logarithmic) Gromov--Witten invariant then corresponds a {\it matrix element} for an operator obtained as an appropriate composition of the $M_l$'s above. The equality between a the Gromov--Witten invariant and the corresponding matrix element  goes through a comparison with the floor diagrams count: by Wick's theorem a matrix element can be naturally evaluated as a weighted sum over Feynman graphs (see Definition \ref{def-Feynmangraph}). In Theorem \ref{thm-operator} we exhibit a weight preserving bijection between the Feynman graphs for a given matrix element, and the floor diagrams for the corresponding Gromov--Witten invariant.


\begin{figure}[h!]
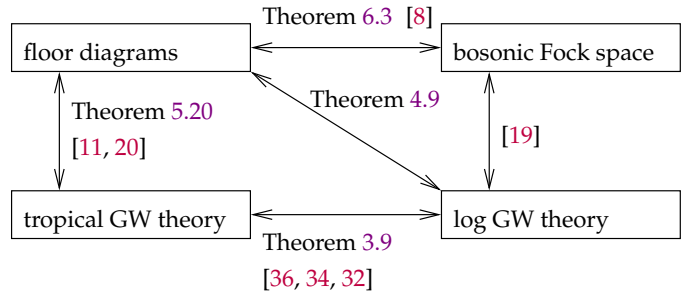
\caption{An overview of the content and background.}\label{fig-results}

\end{figure}

\subsection{Context and Motivation}

This work provides an extension and unification of several previous lines of investigation on the subject. 
Correspondence theorems between tropical curve counts and primary Gromov--Witten invariants of surfaces -- those with only point conditions and no descendant insertions -- were established by Mikhalkin, Nishinou--Siebert, and Gathmann--Markwig in \cite{Mi03, NS06, GM052}; the tropical descendant invariants in genus $0$ was first investigated by Markwig--Rau~\cite{MR08}, and correspondence theorems were established independently, using different techniques, by A. Gross~\cite{Gro15} and by Mandel--Ruddat~\cite{MR16}. Tropical descendants have also arisen in aspects of the SYZ conjecture~\cite{Gro10,O15}.

Cooper and Pandharipande pioneered a Fock space approach to the Severi degrees of $\PP^1\times \PP^1$ and $\PP^2$ by using degeneration techniques~\cite{CP12}. Block and G\"ottsche generalized their work to a broader class of surfaces  (the $h$-transverse surfaces, see for instance~\cite{AB13} and~\cite{BM08}), and to refined curve counts, via quantum commutators on the Fock space side~\cite{BG14}. 
Both cases deal only with primary invariants. Block and G\"ottsche assign an operator on the Fock space to point insertions, and observe the connection between floor diagrams and Feynman graphs. We generalize their operator to a family of operators, one for each descendant insertion, and notice that the  operators can be written with summands naturally corresponding to the possible {\it sizes of the floor} (see Definition \ref{def-operator})  containing a particular descendant insertion. In the primary case, there were only floors of size $0$ (elevators) or $1$(floors), and hence the operator had two terms.

Section \ref{sec-floor} contains a brief summary of how floor diagrams came to be employed for these types of enumerative problems (Subsection \ref{workflo}). This discussion follows our definition of floor diagrams (Definition \ref{def-floor}), to explain and motivate some of the minor combinatorial tweaks we made in order to adapt to the current geometric context.
 \vspace{0.2cm}
 
 It is at this point a well understood philosophy that correspondence theorems between classical and tropical enumerative invariants are based on the fact that tropical curves encode the combinatorics of possible degenerations of the classical objects. The decomposition formula for the Gromov--Witten invariants of simple normal crossings degenerations allows us to equip tropical curves with a virtual multiplicity, and state the correspondence theorem between the virtual counts of tropical and algebraic curves~\cite{ACGS}. 
 
 An appealing feature of the generality provided by the logarithmic setup is that it produces a formula from which one can witness the collapsing of geometric inputs in different settings to give rise to a purely combinatorial theory. In genus $0$, the descendant contributions collapse into closed combinatorial formulas. Conceptually, this is because the intersection theory of the space of genus $0$ logarithmic maps is essentially captured by the intersection theory on a particular toric variety, see~\cite{R15b}. Without descendants but still in higher genus, there is a different collapsing -- on a surface, one can degenerate in such a way that all the algebraic inputs are $1$ up to multiplicity -- the multiplicity can be detected combinatorially, leading to Mikhalkin's formula (Section~\ref{rem-gen}). 

 A drawback of the logarithmic approach to this enumerative problem is that there is not yet a formula expressing the virtual multiplicity of a tropical curve in terms of vertex multiplicities, although such a formula is expected to exist\footnote{Since this paper first entered public circulation there has been addition progress on logarithmic degeneration formulas, but the resulting vertex formulas are still complicated and not immediately implementable~\cite{R19}.}. In lieu of it, we present two options. The first is to change our geometric setup to the older relative maps geometry. The second is to prove a vertex multiplicity formula for special choices of configurations of points. We do this by using tropical arguments to limit the types of tropical curves that can contribute to horizontally stretched descendant constraints. In both cases, the floor diagram connects the invariants to the Fock space.

  Restricting our attention to the study of invariants of \textbf{Hirzebruch surfaces} is a stylistic choice, as we strived to write a paper that communicates the various connections we explore, rather than making the most general statements possible.  Results of Section \ref{sec-tropdesc} could as well be formulated for any toric surface, results of Sections \ref{sec-floor},~\ref{troflo}, and \ref{sec-fock} for any toric surface dual to an $h$-transverse lattice polygon, see \cite[Section 2.3]{BG14}.   

 This paper is a sequel to the authors' work in~\cite{CJMR}, in which the relationship between tropical curves, Fock spaces, and degeneration techniques was studied for target curves, combining Okounkov and Pandharipande's seminal work in~\cite{OP06}, with the tropical perspective on the enumerative geometry of target curves~\cite{CJM1,CJM2,CMR14b,CMR14a}. We refer the reader to~\cite{CJMR} for a more detailed discussion of the history of the target curve case.
 
The paper is organized as follows. In Section \ref{sdhs} we present some basic facts about the geometry of Hirzebruch surfaces, and introduce logarithmic and relative stationary descendant Gromov--Witten  invariants. Section \ref{sec-tropdesc} introduces the tropical theory of descendant stationary invariants of Hirzebruch surfaces, and proves the correspondence theorem with the logarithmic theory. 
In Section \ref{sec-floor} we define our version of decorated floor diagrams,  explain the connection with the previous notions in the literature, and then  compare floor diagram counts with the relative theory, as an application of the degeneration formula. In section \ref{troflo}, we develop a vertex multiplicity formula for floor decomposed tropical maps. We then provide a correspondence theorem relating the count of floor diagrams with the tropical theory, using a combinatorial argument and keeping track of the local vertex multiplicities.  
Section \ref{sec-fock} provides a brief and hopefully friendly introduction to the Fock space, and then proves the equivalence between floor diagram counts and matrix elements for specific operators in the Fock space. 
 
 \subsection*{Acknowledgements}
The work presented in this text was initiated during a Research in Pairs program at the Oberwolfach Institute for Mathematics and completed while various subsets of the authors were at the Fields Institute and the American Institute of Mathematics. We thank the institutes for their hospitality and for providing excellent working conditions. D.R. was a student at Yale University and a member at the Institute for Advanced Study during important phases of this work, and acknowledges friends and colleagues at these institutions for support and encouragement. 

\subsection*{Funding} R.C. acknowledges support by NSF grant FRG-1159964 and Simons collaboration grant 420720. H.M. acknowledges support by DFG-grant MA 4797/6-1. D.R. acknowledges support by NSF grants CAREER  DMS-1149054 (PI: Sam Payne) and DMS-1128155 (Institute for Advanced Study).

\vspace{0.2cm}
\setcounter{tocdepth}{1}
\tableofcontents

\section{Relative and logarithmic descendants}
\label{sdhs}

We study two closely related algebro-geometric curve counting theories attached to a Hirzebruch surface -- the relative and logarithmic Gromov--Witten invariants with stationary descendants. 

For $k \geq 0$, the \textbf{Hirzebruch surface} $\mathbb{F}_k$ is defined to be the surface $\mathbb{P}(\mathcal{O}_{\PP^1}\oplus \mathcal{O}_{\PP^1}(k))$; 
it is a smooth projective toric surface. The $1$-skeleton of its fan $\Sigma_k$ is given by the four vectors $e_1, \pm e_2,  -e_1+ke_2$.  The $2$-dimensional cones are spanned by the consecutive rays in the natural counterclockwise ordering. The zero section $B$, the infinity section $E$, and the fiber $F$ have intersections  
\[
B^2=k,\ \ E^2=-k, \ \ BF=EF=1, \ \ \mbox{and} \ F^2= BE = 0.
\] 
The Picard group of $\mathbb{F}_k$ is isomorphic to $\ZZ\times \ZZ$ generated by the classes of $B$ and $F$. In particular, we have $E = B-kF$. A curve in $\mathbb{F}_k$ has \textbf{bidegree $(a,b)$} if its class is $aB+bF$. The polygon depicted in Figure \ref{fig-Hirz} defines $\mathbb{F}_k$ as a projective toric surface polarized by an $(a,b)$ curve. 

\begin{figure}[h!]
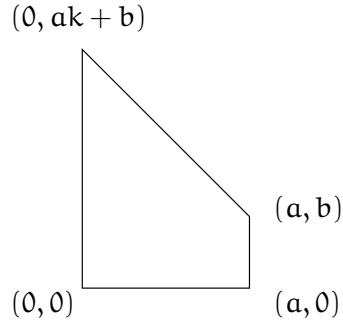
\caption{The polygon defining the Hirzebruch surface $\mathbb{F}_k$ as a toric surface \rc{embedded in projective space} with hyperplane section the class of a curve of bidegree $(a,b)$. The vertical sides corresponds to the sections $B$ (left) and $E$ (right).}\label{fig-Hirz}
\end{figure}

We study the virtual enumerative invariants of curves in Hirzebruch surfaces that have prescribed special contact orders with the zero and infinity sections, \rc{and} generic intersection with the invariant fibers. This numerical data is encoded in terms of the \textit{Newton fan}. 

\begin{definition}\label{def-newtonfan}
A \textbf{Newton fan} is a sequence $\delta=\{v_1,\ldots,v_k\}$ of vectors $v_i\in \ZZ^2$ satisfying $$\sum_{i=1}^k v_i=0.$$ 
If $v_i = (v_{i1},v_{i2})$, then the positive integer $w_i=\gcd(v_{i1}, v_{i2})$ (resp. the vector
$\frac{1}{w_i}v_i$) is called the \textbf{expansion factor} (resp. the
\textbf{primitive direction}) of $v_i$.
We use the notation $$\delta=\{v_1^{m_1},\ldots, v_k^{m_k}\}$$ to indicate
that the vector $v_i$ appears $m_i$ times in $\delta$.
\end{definition}

For a Newton fan $\delta$, one can
construct a polarized toric surface, identified by the \textbf{dual polygon} $\Pi_\delta$ in $\RR^2$, in the following
way: for each primitive integer direction $(\alpha,\beta)$ in $\delta$,
we consider the vector $w(-\beta,\alpha)$, where $w$ is the sum of the expansion factors of all vectors in $\delta$ with primitive integer direction
$(\alpha,\beta)$. Up to translation,
$\Pi_\delta$ is the unique (convex, positively oriented) polygon whose
oriented edges are exactly the vectors $w(-\beta,\alpha)$.



\begin{notation}\textit{(Discrete data)}\label{not-delta}
Fix a Hirzebruch surface $\mathbb{F}_k$. The following discrete conditions govern the enumerative geometric problems we study throughout the paper:
\begin{itemize}
	\item A positive integer $n$;
	\item Non-negative integers $g, a, k_1,\ldots,k_n, n_1,n_2$ ;
	\item A vector ${\underline{\phi}} = (\varphi_1, \ldots, \varphi_{n_1}) \in (\ZZ \setminus 0)^{n_1}$;
	\item A vector ${\underline{\mu}} = (\mu_1, \ldots, \mu_{n_2}) \in (\ZZ \setminus 0)^{n_2}$.
	\item We assume that ${\underline{\phi}}$ and ${\underline{\mu}}$ are non-decreasing sequences.
	\item We denote by $({\underline{\phi}}^+,{\underline{\mu}}^+)$ the positive entries of $({\underline{\phi}},{\underline{\mu}})$, and by $({\underline{\phi}}^-,{\underline{\mu}}^-)$ the negative ones.
\end{itemize}
Further, the following two equations must be satisfied:
\begin{eqnarray}
\sum_{i=1}^{n_1} \varphi_i+\sum_{i=1}^{n_2} \mu_i+ka=0; \label{curveclass}\\
n_2+2a+g-1=n+\sum_{j=1}^n k_j. \label{dimension}
\end{eqnarray}
\end{notation}

\begin{figure}
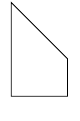
\caption{A visualization of the discrete data from Notation \ref{not-delta}: $({\underline{\phi}}^-,{\underline{\mu}}^-)$ are the contact orders of (fixed and non-fixed) points on the section $B$,  $({\underline{\phi}}^+,{\underline{\mu}}^+)$ on the section $E$ , $a$ the number of intersection points with a fiber. $n$ is the number of marked points, $g$ is the genus and the $k_i$ are the powers for the descendant we impose at the point $p_i$.}\label{fig-discretedata}
\end{figure}

\subsection{Logarithmic invariants}\label{sec:vei} The first enumerative geometric problem we introduce is stationary, descendant, logarithmic Gromov--Witten invariants of $\mathbb{F}_k$, which morally count curves in $\mathbb{F}_k$ with prescribed tangency conditions along the boundary, and 
satisfying some further geometric constraints, called {\bf descendant insertions} (see Section \ref{sec:vei}),  at a number of fixed points in the interior of the surface.
In this context, $g$ is the arithmetic genus of the curves being counted, $n$ is the number of ordinary marked points on the curves, and the $k_i$ are the degrees of the descendant insertions at each point. The sequences $({\underline{\phi}},{\underline{\mu}})$ identify   a curve class  in  $H_2(\mathbb{F}_k, \ZZ)$, as well as the required tangency with the toric boundary, as we now explain. 

The tuple $({\underline{\phi}},{\underline{\mu}})$ determines the curve class 
\begin{equation}\beta = aB + \left(\sum_{\varphi_i \in {{\underline{\phi}}^+}} \varphi_i +\sum_{\mu_i \in {{\underline{\mu}}^+} }\mu_i\right)F.
\end{equation}
The compatibility condition \eqref{curveclass} ensures that $\beta$ is an effective, integral curve class in $H_2(\mathbb{F}_k,\ZZ)$.

The Newton fan 
\begin{equation}
\delta_{({\underline{\phi}},\underline{\mu})}:=\{(0,-1)^a, (k,1)^a, \varphi_1\cdot (1,0), \ldots,\varphi_{n_1}\cdot (1,0), \mu_1\cdot (1,0), \ldots, \mu_{n_2}\cdot (1,0)\}.
\end{equation}
encodes contact orders a curve may have with the toric boundary of $\mathbb{F}_k$. Such a curve is necessarily of class $\beta$.

We count curves with contact orders $|\varphi_i|$ for $\varphi_i<0$ (resp. $\varphi_i>0$) with the zero (resp. infinity) section at fixed points, and contact orders $|\mu_i|$ for $\mu_i<0$ (resp. $\mu_i>0$) with the zero (resp. infinity) section at arbitrary points.

Logarithmic stable maps and logarithmic Gromov--Witten invariants were developed in ~\cite{AC11,Che10,GS13}. 
There is a moduli stack with logarithmic structure
\[
\overline{M}^{\mathsf{log}}_{g,n+n_1+n_2}(\mathbb{F}_k, \delta_{({\underline{\phi}},\underline{\mu})}),
\]
and a map from a logarithmic scheme $S$ to the moduli stack is equivalent to a diagram 
\[
\begin{tikzcd}
\mathscr C \arrow{d}\arrow{r}{f} & \mathbb F_k, \\
S, 
\end{tikzcd}
\]
where $\mathscr C$ is a family of connected marked genus $g$ nodal logarithmic curves and $f$ is a map of logarithmic schemes, whose underlying map is stable in the usual sense. The contact orders with the toric boundary are specified by the Newton fan $\delta_{({\underline{\phi}},\underline{\mu})}$. As a matter of convention, we mark the points of contact with the zero and infinity sections, and do not mark the points of contact with the torus-invariant fibers, where the behavior requested is transverse. 

This moduli space is a proper Deligne-Mumford stack equipped with a virtual fundamental class which we denote $[1]^\mathsf{log}$ in degree $(g-1) +2a +n+n_1+n_2$. For each of the first $n$ marked points, which carry trivial contact orders, there are evaluation morphisms
\[
ev_i: \overline{M}^{\mathsf{log}}_{g,n+n_1+n_2}(\mathbb{F}_k, \delta_{({\underline{\phi}},\underline{\mu})})\to \mathbb F_k
\]
The points marking the contact points with the zero and infinity sections give rise to evaluation morphisms 
\[
\widehat{ev_i}: \overline{M}^{\mathsf{log}}_{g,n+n_1+n_2}(\mathbb{F}_k, \delta_{({\underline{\phi}},\underline{\mu})})\to \PP^1.
\] 
Here, the target $\PP^1$ is the the zero section $B$ for negative entries of $\underline{\phi}$ or $\underline{\mu}$, and the infinity section $E$ for positive entries.

For each of the first $n$ marks (i.e. those with trivial contact order) there is a cotangent line bundle, whose first Chern class is denoted $\psi_i$.

\begin{definition}\label{def-logdescGWI}
Fix a Hirzebruch surface $\mathbb{F}_k$ and discrete data as in Notation \ref{not-delta}. 

The \textbf{stationary descendant log Gromov--Witten} invariant is defined as the following intersection number on $\overline{M}^{\mathsf{log}}_{g,n+n_1+n_2}(\mathbb{F}_k, \delta_{({\underline{\phi}},\underline{\mu})})$:
\begin{equation}\label{eq-lgw}
\langle ({\underline{\phi}}^-,{\underline{\mu}}^-)| \tau_{k_1}(pt)\ldots\tau_{k_n}(pt)|({\underline{\phi}}^+,{\underline{\mu}}^+)\rangle^\mathsf{log}_{g}=
\int_{[1]^\mathsf{log}} \prod_{j=1}^n \psi_j^{k_j}ev_j^\ast([pt])\prod_{i=n+1}^{n+n_1}\widehat{ev}_i^\ast([pt])
\end{equation}

Condition $\eqref{dimension}$ comes from equating the expected dimension of the moduli space with the codimension of the intersection cycle, and hence it is a necessary condition for Equation \eqref{eq-lgw} to be non-zero. 

%
\end{definition}

\subsection{Relative invariants}\label{sec: relative-gwi} Closely related to the logarithmic invariants studied above are \textit{relative} Gromov--Witten invariants. Let $\mathbb F_k$ continue to denote the Hirzebruch surface, now considered as a pair $(\mathbb F_k, \rc{B+E})$, \rc{ the divisor consisting of the disjoint union of the zero and infinity sections.} Fix discrete data  as in Notation~\ref{not-delta}. There is a moduli space
\[
\overline{M}^{\mathsf{rel}}_{g,n+n_1+n_2}(\mathbb{F}_k, \delta_{({\underline{\phi}},\underline{\mu})}),
\]
parameterizing families of maps to expansions
\[
\mathscr C\to S_1\cup \cdots \cup S_m \to \mathbb F_k,
\]
where each $S_i$ is a copy of the Hirzebruch surface $\mathbb F_k$, where the zero section of $S_i$ is glued to the infinity section of $S_{i+1}$. As before, the curve $\mathscr C$ carries $n+n_1+n_2$ markings, and the contact orders  \rc{at fixed points of the zero section of $S_1$ (resp. infinity section of $S_m$)}  are given by $|\varphi_i|$ for $\varphi_i<0$ (resp. $\varphi_i>0$), and \rc{at arbitrary points of the zero section of $S_1$ (resp. infinity section of $S_m$) are specified to be $|\mu_i|$ for $\mu_i<0$ (resp. $\mu_i>0$)}. See~\cite{GV05,Li02} for additional details on maps to expansions.  

Both the relative and logarithmic invariants are virtual counts \rc{for} the same enumerative problem. The main difference between the \rc{logarithmic and relative}  setups is that contact orders are not prescribed with the \rc{torus invariant} fibers in the \rc{latter}. The two theories are closely related, since in the logarithmic case, non-transverse contact orders are only imposed along the invariant sections of the surface. 

\begin{remark}
In the remainder of this paper, the relative and logarithmic theories will not interact, except in that they both relate to floor diagrams. However, a referee has shared with us an outline for explaining why the two invariants agree, which we include gratefully. First we note that for stationary descendant invariants, an application of the logarithmic decomposition formula~\cite{ACGS} and combinatorial considerations reduce the problem to the case when the toric surface is $\mathbb P^1\times \mathbb P^1$ and the invariants considered are of genus zero. The statement of~\cite[Theorem~5.1]{MR16} and the arguments in its proof then allow the invariant fibers to be removed from the logarithmic structure.
\end{remark}

There is once again a virtual fundamental class in the homology of $
\overline{M}^{\mathsf{rel}}_{g,n+n_1+n_2}(\mathbb{F}_k \delta_{({\underline{\phi}},\underline{\mu})}),$ in degree $(g-1)+2a+n_1+n_2+n$. The spaces of relative stable maps come equipped with evaluation morphisms. The following definition is analogous to Definition~\ref{def-logdescGWI}.

\begin{definition}\label{def-reldescGWI}
The \textbf{stationary descendant relative Gromov--Witten invariant} is defined by
\begin{equation}\label{eq-gw}
\langle ({\underline{\phi}}^-,{\underline{\mu}}^-)| \tau_{k_1}(pt)\ldots\tau_{k_n}(pt)|({\underline{\phi}}^+,{\underline{\mu}}^+)\rangle^{\mathsf{rel}}_{g}=
\int_{[1]^\mathsf{rel}} \prod_{j=1}^n \psi_j^{k_j}ev_j^\ast([pt])\prod_{i=n+1}^{n+n_1}\widehat{ev}_i^\ast([pt]).
\end{equation}
\end{definition}


\section{Tropical descendants}\label{sec-tropdesc}

\subsection{Tropical preliminaries} An \textbf{(abstract) tropical curve} is a connected metric graph $\Gamma$ with unbounded rays or ``ends'' and a genus function $g:\Gamma\rightarrow \NN$ which is nonzero only at finitely many points. Locally around a point $p$, $\Gamma$ is homeomorphic to a star with $r$ half-rays. 
The number $r$ is called the \textbf{valence} of the point $p$ and denoted by $\val(p)$.
We require that there are only finitely many points with $\val(p)\neq 2$. We require that the set of all points of nonzero genus or valence larger than $2$ is non empty.
\rc{A finite set of points containing  (but not necessarily equal to the set of) all points of nonzero genus or valence larger than $2$ may be chosen; its elements  are called  \textbf{vertices}. }
By abuse of notation, the underlying graph with this vertex set is also denoted by $\Gamma$. 
Correspondingly, we can speak about \textbf{edges} and \textbf{flags} of $\Gamma$. A flag is a tuple $(V,e)$ of a vertex $V$ and an edge $e$ with $V\in \partial e$. It can be thought of as an element in the tangent space of $\Gamma$ at $V$, i.e.\ as a germ of an edge leaving $V$, or as a half-edge (the half of $e$ that is attached to $V$). Edges which are not ends have a finite length and are called \textbf{bounded edges}.

A \textbf{marked tropical curve} is a tropical curve such that some of its ends are labeled. An isomorphism of a tropical curve is a homeomorphism respecting the metric, the markings of ends, and the genus function. The \textbf{genus} of a tropical curve is the first Betti number $b^1(\Gamma)$ plus the genera of all vertices. A curve of genus $0$ is called \textbf{rational}. 

The \textbf{combinatorial type} of a tropical curve is obtained by dropping the information on the metric. 

Let $\Sigma$ be a polyhedral decomposition of $\RR^2$. 

\begin{definition}
A \textbf{tropical stable map to $\Sigma$} is a tuple $(\Gamma,f)$ where $\Gamma$ is a marked abstract tropical curve and $f:\Gamma\to \Sigma$ is a piecewise integer-affine map of polyhedral complexes satisfying:
\begin{itemize}
\item On each edge $e$ of $\Gamma$, $f$ is of the form $$t\mapsto a+t\cdot v \mbox{ with } v\in \ZZ^2,$$ where we parametrize $e$ as an interval of size the length $l(e)$ of $e$. The vector $v$, called the \textbf{direction}, arising in this equation is defined up to sign, depending on the starting vertex of the parametrization of the edge. We will sometimes speak of the direction of a flag $v(V,e)$. If $e$ is an end we use the notation $v(e)$ for the direction of its unique flag.
\item The \textbf{balancing condition} holds at every vertex, i.e.\ 
$$\sum_{e \in \partial V} v(V,e)=0.$$
\item The \textbf{stability condition} holds, i.e. for every $2$-valent vertex $v$ of $\Gamma$, the star of $v$ is not contained in the relative interior of any single cone of $\Sigma$.
\end{itemize}
\end{definition}
For an edge with direction $v=(v_1,v_2) \in \ZZ^2$, 
we call $w=\gcd(v_1,v_2)$ the \textbf{expansion factor} and $\frac{1}{w}\cdot v$ the \textbf{primitive direction} of $e$.

An isomorphism of tropical stable maps is an isomorphism of the underlying tropical curves respecting the map. The \textbf{degree} of a tropical stable map is the Newton fan given as the multiset of directions of its ends.
The \textbf{combinatorial type} of a tropical stable map is the data obtained when dropping the metric of the underlying graph. More explicitly, it consists of the data of a finite graph $\Gamma$, and (1) for each vertex $v$ of $\Gamma$, the cone $\sigma_v$ of $\Sigma$ to which this vertex maps, and (2) for each edge $e$ of $\Gamma$, the expansion factor and primitive direction of $e$.

Note that in practice, the precise polyhedral decomposition plays a limited role, and we will often drop this from the discussion, simply referring to the maps by the notation $[\Gamma\to \RR^2]$. 

\begin{convention}
We consider tropical stable maps to Hirzebruch surfaces, i.e.\ the degree is a Newton fan dual to the polygons of Figure \ref{fig-Hirz}. Furthermore, we require the vertical and diagonal ends to be non-marked and of expansion factor $1$. The horizontal ends can have any expansion factor, and are marked.
\end{convention}

In what follows, we fix conditions for tropical stable maps --- the degree, the genus, point conditions, high valency (descendant) conditions, and end conditions --- and then count tropical stable maps satisfying the conditions, with multiplicity. We consider degrees containing integer multiples of $(1,0)$. An end whose direction vector is a multiple of $(1,0)$ is mapped to a line segment of the form $\{(a,b)+t\cdot (\pm 1,0)\}$, where $(a,b)\in \RR^2$. The unique $b$ appearing here is the \textbf{$y$-coordinate} of the respective end. Our end conditions fix some of the $y$-coordinates of ends.

\begin{definition}\label{def-tropdescGWI}
Fix discrete invariants as in Notation \ref{not-delta}.
Let $$\Delta=\delta_{({\underline{\phi}},{\underline{\mu}})}\cup \{0^n\}$$ identify a degree for tropical stable maps. 
Fix $n$ points $p_1,\ldots,p_n\in \RR^2$ in general position, and two sets $E_0$ and $E_\infty$ of pairwise \rc{distinct} real numbers together with bijections $E_0 \to \{\varphi_i|\varphi_i<0\}$ (resp.\ $E_\infty \to \{\varphi_i|\varphi_i>0\}$).

The \textbf{tropical descendant Gromov--Witten invariant} $$\langle ({\underline{\phi}}^-,{\underline{\mu}}^-)| \tau_{k_1}(p_1)\ldots\tau_{k_n}(p_n)|({\underline{\phi}}^+,{\underline{\mu}}^+)\rangle_{g}^{\trop}$$ is the weighted number of marked tropical stable maps $(\Gamma,f)$ of degree $\Delta$ and genus $g$ satisfying:
\begin{itemize}
 \item For $j=1, \ldots, n$, the marked end $j$ is contracted to the point $p_j\in \RR^2$.
 \item The end $j$ is adjacent to a vertex $V$ in $\Gamma$ of valence $\val(V)= k_j+3-g(V)$.
 \item $E_0$ and $E_\infty$ are the $y$-coordinates of ends marked by the set $\phi$.
\end{itemize}
Each such tropical stable map is counted with \textbf{multiplicity} $\frac{1}{\mathrm{Aut}(f)}m_{(\Gamma,f)}$, to be defined in Definition \ref{def: virtual-multiplicity}.
\end{definition}


\subsection{Superabundance and rigid curves}  \rc{The set of tropical stable maps of a fixed combinatorial  type} can be parametrized by a polyhedral cone in a real vector space, see~\cite{GS13,NS06}, as well as the closely related~\cite{GM051, GM053, Mi07}. The \textbf{expected dimension} of the cone associated to the type of a map $(\Gamma,f)$ is 
\[
\#\{\mbox{ends}\}+b^1(\Gamma)-1-\sum k_i-\sum_V (\val(V)-3) = \#\{\mbox{bounded edges}\}-2b^1(\Gamma)+2.
\]
When a combinatorial type has this expected dimension, it is said to be \textbf{non-superabundant}. 
 In superabundant cases, there may be nontrivial families even when the expected dimension is zero, so we \rc{introduce the}  notion of rigidity to reduce to a finite combinatorial count.


%

\begin{definition}
Choose general points $p_1,\ldots, p_n\in \RR^2$, a degree, genus, incidence, and descendant constraints defining a tropical descendant Gromov--Witten invariant. Let $(\Gamma,f)$ be a tropical stable map satisfying these chosen constraints. The map $(\Gamma,f)$ is said to be \textbf{rigid} if $(\Gamma,f)$ is not contained in any nontrivial family of tropical curves having the same combinatorial type. 
\end{definition}

The following result follows from a simple adaptation of the proof of~\cite[Lemma 4.20]{Mi03}.

\begin{lemma}\label{lem-gammaminuspoint} Let $(\Gamma,f)$ be a rigid stable map satisfying the conditions of Definition \ref{def-tropdescGWI}. Then every connected component of $\Gamma$ minus the marked ends is rational and contains exactly one \rc{non-fixed} end.
\end{lemma}

\subsection{The virtual multiplicity of a rigid tropical curve}\label{sec: virtual-multiplicity} Let 
\[
f:\Gamma\to\RR^2
\]
be a rigid tropical stable map contributing to a tropical descendant Gromov--Witten invariant. Assume that all vertices of $\Gamma$ map to integer points in $\mathbb R^2$. Note that since edge expansion factors are rational, this is always possible after a translation and dilation. We define the multiplicity of $(\Gamma,f)$. This is done in two steps. Following~\cite[Section~6.3.1]{ACGS} let $\mathbb F_k^\dagger$ denote the product $\mathbb F_k\times\spec(\NN\to \CC)$, where the latter is the standard logarithmic point. {Consider} a minimal logarithmic stable map
\[
\begin{tikzcd}
C\arrow{r}\arrow{d} & \mathbb F^\dagger_k \arrow{d} \\
S \arrow{r} & \spec(\mathbb N\to \CC).
\end{tikzcd}
\]
A logarithmic curve over a standard logarithmic point determines a tropical curve $\Gamma_C$. Specifically, we take the underlying graph of $\Gamma_C$ to be the dual graph of $C$. Attached to each node $q$ of $C$ is a deformation parameter $\delta_q\in \NN$. We take the length of the corresponding edge $e_q$ to be the value $\delta_q$. As explained in~\cite[Section 2]{ACGS}, by dualizing the maps of monoids underlying the map of logarithmic schemes $C\to \mathbb F_k^\dagger$, we obtain a tropical map
\[
\Gamma_C\to \RR^2.
\]

Fix $(\Gamma,f)$ a tropical stable map.  Consider logarithmic stable maps $C\to \mathbb F_k^\dagger$, equipped with an edge contraction from the tropicalization $\Gamma_C\to \RR^2$ to $(\Gamma,f)$. That is, the logarithmic maps are at least as degenerate as dictated by the fixed tropical map $(\Gamma,f)$. However, the marking by $(\Gamma,f)$ is an additional rigidifying datum. In~\cite[Section 4]{ACGS}, it is shown that such maps are parameterized by an algebraic stack with a finite map
\[
\overline M_{(\Gamma,f)}\to \overline{M}^{\mathsf{log}}_{g,n+n_1+n_2}(\mathbb F_k^\dagger,\delta_{\underline\phi,\underline\mu}).
\]

We may now define the virtual multiplicity of $(\Gamma,f)$. For simplicity, we only treat the case without fixed points on the boundary. It is a notational exercise to extend this to the general case. Given a standard logarithmic point 
\[
p^\dagger: \spec(\NN\to \CC)\to \mathbb F_k^\dagger,
\] 
there is an associated tropical point in $\RR^2$. Let $U_\sigma$ be the affine invariant open to which the underlying scheme theoretic point maps. Let $S_\sigma$ be the associated dual cone of characters that extend to regular functions on $U_\sigma$. By the definition of a logarithmic morphism, there is an induced map
\[
S_\sigma\to \NN,
\]
which gives an integral point $p\in \RR^2$. We refer to $p^\dagger$ as a logarithmic lifting of $p$. 

Fix discrete data. Choose general points $p_1,\ldots, p_n\in \RR^2$, and let $(\Gamma,f)$ be a rigid tropical stable map passing through the points $p_i$ contributing to a tropical descendant Gromov--Witten invariant. Choose a logarithmic lifting $\underline p^\dagger =(p_1^\dagger,\ldots, p_n^\dagger)$ of $(p_1,\ldots, p_n)$, which as above, is an $n$-tuple of logarithmic sections of $\mathbb F_k^\dagger$. Evaluation at the $n$ markings of the curve gives rise to a morphism
\[
ev: \overline M_{(\Gamma,f)}\to (\mathbb F_k^\dagger)^n.
\]
We first create a moduli space of logarithmic curves that pass through these logarithmic points by choosing evaluations:
\[
\overline M_{(\Gamma,f)}(\underline p^\dagger) := \overline M_{(\Gamma,f)}\times_{(\mathbb F_k^\dagger)^n} \underline p^\dagger.
\] 
As explained in~\cite[Section~6.3.2]{ACGS}, this moduli space comes equipped with a virtual fundamental class. This moduli space will still have positive virtual dimension, so we cut it down using the descendant classes. 

\begin{definition}\label{def: virtual-multiplicity}
For a rigid tropical curve $(\Gamma,f)$ contributing to a tropical descendant Gromov--Witten invariant, define its \textbf{virtual multiplicity} to be
\[
m_{(\Gamma,f)}:= \int_{[\overline M_{(\Gamma,f)}(\underline p^\dagger)]^{\mathrm{vir}}} \prod_{j=1}^n \psi_j^{k_j}.
\]
A non-rigid tropical curve $(\Gamma,f)$ is defined to have multiplicity $0$.
\end{definition}

A requisite for the above definition is the fact that the multiplicity defined above is independent of the choice of logarithmic lifting. This follows from the logarithmic deformation invariance of the virtual fundamental class~\cite[Theorem~0.3]{GS13} and~\cite[Appendix~A]{MR16}. This completes the definition of the tropical descendant Gromov--Witten invariant. 

\subsection{When does the virtual multiplicity collapse?}\label{rem-gen}
The multiplicity in Definition~\ref{def: virtual-multiplicity} is difficult to compute in practice, which limits the utility of the logarithmic decomposition. There has been encouraging progress in providing explicit calculation schemes for these multiplicities, see for instance~\cite{MR19}. In Section~\ref{troflo}, we  observe that after specializing the point conditions to horizontally stretched position (leading to floor decomposed tropical maps), it is indeed possible to write the multiplicity $m_{(\Gamma,f)}$ in terms of local multiplicities attached to vertices, yielding a \textit{degeneration} formula, as opposed to only a decomposition. We point out two previous instances in which this was already known~\cite{MR08,Mi03}.

\begin{enumerate}
\item \textbf{If all $\psi$-powers are $0$, i.e.\ $k_1=\ldots=k_n=0$:} the valency condition implies that the vertex adjacent to end $i$ is trivalent and of genus $0$. Since the end $i$ is contracted, the image of a neighbourhood of this vertex just looks like an edge passing through $p_i$. We thus count plane tropical curves passing through the points (and possibly with some fixed $y$-coordinates for the ends). An example can be found in Example \ref{ex-tropstablemap}, see Figure \ref{fig-exki0}. They are counted with multiplicity equal to the product of the normalized areas of the  triangles in the dual subdivision (notice that all vertices are trivalent and of genus $0$ for dimension reasons).
In case of fixed $y$-coordinates, the product above has to be multiplied in addition with $\prod_e\frac{1}{w(e)}$, where the product goes over all fixed ends $e$ and $w(e)$ denotes their expansion factor \cite{GM052}. That all local Gromov--Witten invariants are $1$ follows from the correspondence theorem proved in~\cite{NS06,R15b}. 
\item \textbf{If the genus $g=0$:}  in \cite{MR08}, \rc{tropical, rational, stationary, descendant invariants are studied. The appropriate tropical maps }
 are counted with multiplicity equal to the product of the normalized areas of the triangles (dual to non-marked vertices) in the dual subdivision as above, with a factor of $\prod_e\frac{1}{w(e)}$ for fixed ends, see~\cite{BGM10}. The correspondence theorem for such invariants is proved in the papers~\cite{Gro15,R15b} and using different methods in~\cite{MR16}. 
 \end{enumerate}

\begin{example}\label{ex-tropstablemap}
We show two examples. The point conditions $p_i\in \RR^2$ are chosen to be in \textbf{horizontally stretched} position, see~\cite[Definition 3.1]{FM09}.
\begin{enumerate}
\item Let $k=1$, $({\underline{\phi}})=(-2,1)$, $({\underline{\mu}})=(-2,-1,1)$. Then $\sum \varphi_i+\sum \mu_i+3\cdot 1=0$, so $a=3$. Let $g=0$, $n=8$, and $k_1=\ldots=k_8=0$. Since $n_2=3$ and $3+2\cdot 3-1=8$, this choice satisfies the condition of Definition \ref{def-tropdescGWI}. Figure \ref{fig-exki0} shows the image of a tropical stable map contributing to $\langle((-2),(-2,-1)) |\tau_{0}(p_1)\ldots\tau_{0}(p_8)|((1),(1))\rangle_{0}^{\trop}$ with multiplicity $72$ (see Section~\ref{rem-gen} (1)). The Figure reflects the image of the map, decorated by some data of the parametrization --- for that reason, the picture indicates a crossing instead of a $4$-valent vertex. We draw the fixed $y$-coordinates as points at the end of an end. Expansion factors bigger one are written next to the edges, so that the direction is visible from the picture.

\begin{figure}
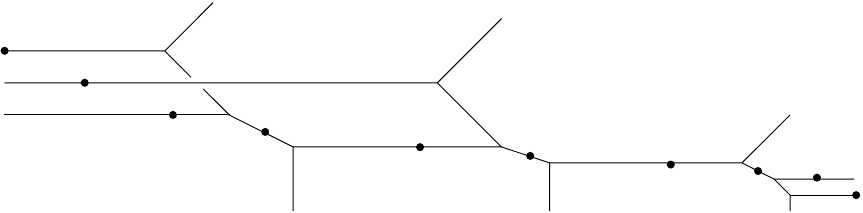
\caption{A tropical stable map to $\mathbb{F}_1$ contributing to $\langle((-2),(-2,-1)) |\tau_{0}(p_1)\ldots\tau_{0}(p_8)|((1),(1))\rangle_{0}^{\trop}$ as in Example \ref{ex-tropstablemap}.}\label{fig-exki0}
\end{figure}

\item As before, let $k=1$, $({\underline{\phi}})=(-2,1)$, $({\underline{\mu}})=(-2,-1,1)$, $a=3$ and $g=0$. Let $n=4$ and $k_1=0$, $k_2=1$, $k_3=3$ and $k_4=0$. Then $3+2\cdot 3-1=4+1+3$, so the condition of Definition \ref{def-tropdescGWI} is satisfied for this choice. Figure \ref{fig-ex} shows a tropical stable map contributing to $\langle ((-2),(-2,-1)) |\tau_{0}(p_1)\tau_1(p_2)\tau_3(p_3)\tau_0(p_4)|((1),(1))\rangle_{0}^{\trop}$ with multiplicity $4$ (see Remark \ref{rem-gen} (2) above). 

\begin{figure}
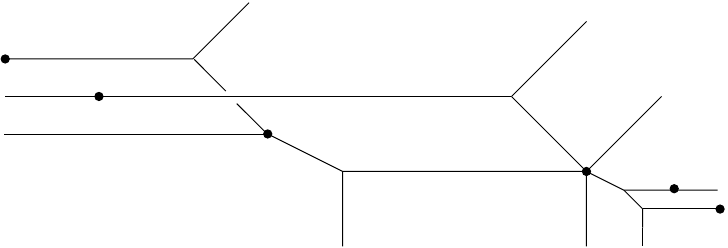
\caption{A tropical stable map to $\mathbb{F}_1$ contributing to the invariant $\langle ((-2),(-2,-1)) |\tau_{0}(p_1)\tau_1(p_2)\tau_3(p_3)\tau_0(p_4)|((1),(1))\rangle_{0}^{\trop}$ as in Example \ref{ex-tropstablemap}.}\label{fig-ex}
\end{figure}

\end{enumerate}
\end{example}

\begin{remark}
The image $f(\Gamma)\subset \RR^2$ of a tropical stable map is a tropical plane curve as considered e.g.\ in \cite{Mi03, RST03}. We assume that the reader is familiar with basic concepts concerning tropical plane curves, in particular their duality to \textbf{subdivisions of the Newton polygon}. In our situation, the image of any tropical stable map contributing to the count above is dual to a subdivision of the polygon attached to the Newton fan $\delta_{({\underline{\phi}},{\underline{\mu}})}$, which defines the Hirzebruch surface $\mathbb{F}_k$ as a projective toric surface with hyperplane section the class of a curve of bidegree $(a,\sum_{i|\varphi_i>0}\varphi_i+\sum_{i|\mu_i>0}\mu_i)$. 
Figure \ref{fig-dual} shows the dual Newton subdivisions of the images of the stable maps of Example \ref{ex-tropstablemap}.

\begin{figure}
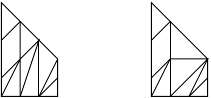
\caption{The dual subdivisions of the images of the stable maps of Example \ref{ex-tropstablemap}.}\label{fig-dual}
\end{figure}

\end{remark}

\begin{theorem}[Correspondence theorem]\label{thm-corres}
Fix a Hirzebruch surface $\mathbb{F}_k$ and discrete data as in Notation \ref{not-delta}.
The tropical stationary descendant log Gromov--Witten invariant 
coincides with its algebro-geometric counterpart 
, i.e.\ we have
 $$\langle ({\underline{\phi}}^-,{\underline{\mu}}^-)|\tau_{k_1}(pt)\ldots\tau_{k_n}(pt)|({\underline{\phi}}^+,{\underline{\mu}}^+)\rangle^{\mathsf{log}}_{g}=\langle ({\underline{\phi}}^-,{\underline{\mu}}^-)|\tau_{k_1}(pt)\ldots\tau_{k_n}(pt)|({\underline{\phi}}^+,{\underline{\mu}}^+)\rangle_{g}^{\trop}$$
\end{theorem} 

\begin{proof}
The proof is a consequence of the decomposition formula for logarithmic Gromov--Witten invariants, due to Abramovich, Chen, Gross, and Siebert~\cite{ACGS}. We explain the geometric setup, and how to deduce the multiplicity above from the formulation in loc. cit. We assume that there are no fixed boundary conditions to lower the burden of the notation; the general case is no more complicated.

Consider the moduli space $\overline{M}^{\mathsf{log}}_{g,n+n_1+n_2}(\mathbb{F}_k, \delta_{({\underline{\phi}},\underline{\mu})})$, and on it, the descendant cycle class given by $\psi_1^{k_1}\cdots \psi_n^{k_n}$. We compute the invariant by degenerating the point conditions, and cutting down the virtual class to
\[
\psi_1^{k_1}\cdots \psi_n^{k_n}\cap [1]^{\mathsf{log}}.
\]
Working over $\spec(\CC(\!(t)\!))$, choose points $p_1$, $\ldots$, $p_n\in T\subset \mathbb F_k$, and assume that the coordinate-wise valuations of these points in $T$, which give points $p_1^{\trop}$, $\ldots$, $p_n^{\trop}$, are in general position in $\RR^2$. Since the tropical moduli space with the prescribed discrete data has only finitely many cones, it follows that there are finitely many rigid tropical stable maps meeting the stationary constraints. Suppose $(\Gamma,f)$ is a tropical stable map with an end $p_i$ incident to a vertex $V$. Since a point $p_i$ must support the descendant class $\psi_i^{k_i}$, a dimension argument forces that the valency of $V$ is $k_V+3-g(V)$. In other words, the tropical curves contributing to the count are precisely the ones outlined in Definition~\ref{def-tropdescGWI}. 

There are finitely many tropical stable maps contributing to this family. Indeed, there are only finitely many combinatorial types of tropical stable maps by the combinatorial finiteness results outlined in~\cite{NS06,GS13}. Moreover, as a type only contributes if it meets the constraints and is rigid, there are only finitely many contributing maps. We enumerate the maps $(\Gamma_1,f_1),\ldots, (\Gamma_r,f_r)$ that contribute to the invariant 
\[
\langle ({\underline{\phi}}^-,{\underline{\mu}}^-)|\tau_{k_1}(p_1^{\trop})\ldots\tau_{k_n}(p_n^{\trop})|({\underline{\phi}}^+,{\underline{\mu}}^+)\rangle_{g}^{\trop}. 
\]
Choose a polyhedral decomposition $\mathscr P$ of $\RR^2$ such that every tropical stable map $f_i^{\trop} $factors through the one-skeleton of $\mathscr P$ and that the fan of unbounded directions of $\mathscr P$ (i.e. the recession fan) is the fan $\Sigma_k$. Note that $\mathscr P$ can always be chosen to be a common refinement of the images of $f_i^{\trop}$. The contact order conditions on the tropical maps ensure that the recession fan is $\Sigma_k$. 

The polyhedral decomposition $\mathscr P$ determines a toric degeneration $\mathscr X$ of $\mathbb F_k$, over $\spec(\CC[\![t]\!])$, see~\cite{NS06}. In~\cite[Appendix~A]{MR16}, the authors show that logarithmic Gromov--Witten invariants are constant in logarithmically smooth families. By using this deformation invariance, we may compute $\langle ({\underline{\phi}}^-,{\underline{\mu}}^-)|\tau_{k_1}(pt)\ldots\tau_{k_n}(pt)|({\underline{\phi}}^+,{\underline{\mu}}^+)\rangle^{\mathsf{log}}_{g}$ on the central fiber of this degeneration, as 
\[
ev^\star(p)\cap \psi_1^{k_1}\cdots \psi_n^{k_n}\cap [1]^{\mathsf{log}},
\]
where 
\[
ev:\overline{M}^{\mathsf{log}}_{g,n+n_1+n_2}(\mathscr X, \delta_{({\underline{\phi}},\underline{\mu})}) \to \mathscr X^n,
\]
is the product of all evaluation morphisms, and $p$ is the specialization of the point $(p_1,\ldots, p_n)$ chosen above.

By applying the decomposition formula for logarithmic Gromov--Witten invariants for point conditions~\cite[Theorem 6.3.9]{ACGS}, this invariant can be written as a sum of the invariants associated to each tropical curve in the manner described. 
\end{proof}

\section{Floor diagrams via the relative theory}\label{sec-floor}

Floor diagrams are connected, via combinatorial manipulations, both to relative descendant Gromov--Witten invariants and tropical descendant Gromov--Witten invariants. This section deals with the former. Floor diagrams naturally organize the computation of a relative descendant via the degeneration formula. A brief discussion of the connection between floor diagrams appearing here and in previous work is found in~Section \ref{workflo}.

\subsection{Floor diagrams} The enumerative geometry studied in this section is the relative descendant Gromov--Witten theory of $\mathbb F_k$, as outlined in Section~\ref{sec: relative-gwi} 

\begin{definition}\label{def-floor}
 Let $D$ be a loop-free connected graph on a linearly ordered vertex set. $D$ has two types of edges: compact edges, composed of two flags (or half-edges), adjacent to different vertices, and unbounded edges, also called ends, with only one flag. $D$ is called a \textbf{floor diagram} for $\mathbb{F}_k$ of degree $({\underline{\phi}},{\underline{\mu}})$
 if:
\begin{enumerate}
\item Three non-negative integers are assigned to each vertex $V$: $g_V$ (called the \textbf{genus} of $V$), $s_V$ (called the \textbf{size} of $V$) and $k_V$ (called the \textbf{$\psi$-power} of $V$). 
\item Each flag may be decorated with a thickening. We require that for each compact edge precisely one of its two half-edges is thickened. 
\item At each vertex $V$, $k_V+2-2s_V-g_V$ adjacent half-edges are thickened.
\item Each edge $e$ comes with an \textbf{expansion factor} $w(e)\in \NN_{>0}$.
\item At each vertex $V$, the signed sum of expansion factors of the adjacent edges (where we use negative signs for edges pointing to the left and positive signs for edges pointing to the right) equals $-k_V\cdot s_V$.
\item The sequence of expansion factors of non-thick ends (where we use negative signs for the ends pointing to the left and positive signs for the ends pointing to the right) is $({\underline{\phi}})$, and  the sequence of expansion factors of thickened ends (with the analogous sign convention) is $({\underline{\mu}})$. 
\item The ends of the graph are marked by the parts of $({\underline{\phi}},{\underline{\mu}})$.

\end{enumerate}

The \textbf{genus} of a floor diagram is defined to be the first Betti number of the graph plus the sum of the genera at all vertices.
\end{definition}

\begin{figure}[tb]
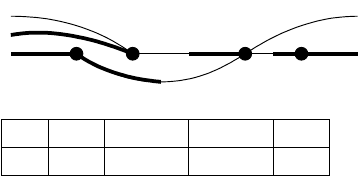
\caption{An example of a floor diagram. The genus at all vertices is $0$.}\label{fig-floordiagram2a}
\end{figure}

\begin{example}
Figure \ref{fig-floordiagram2a} shows  a floor diagram for $\mathbb{F}_1$ of degree $((-2,1),(-2,-1,1))$ and genus $0$.
\end{example}

\begin{definition}\label{def-floormult}
Given  a floor diagram for $\mathbb{F}_k$, let $V$ be a vertex of genus $g_V$, size $s_V$ and with $\psi$-power $k_V$. Let $({\underline{\phi}}_V,{\underline{\mu}}_V)$ denote the expansion factors of the flags adjacent to $V$; the first sequence encodes the normal half edges, the second the thickened ones. We define the multiplicity $\mult(V)$ of $V$ to be the one-point stationary relative descendant invariant $$\mult(V)=\langle ({\underline{\phi}}^-_V,{\underline{\mu}}^-_V)|\tau_{k_V}(pt)|({\underline{\phi}}^+_V,{\underline{\mu}}^+_V)\rangle^{\mathsf{rel}}_{g_V}. $$
\end{definition}

\begin{definition}[Floor multiplicity for relative geometries]\label{def-floorcount}
Fix discrete data as in Notation \ref{not-delta}.
We  define: $$\langle ({\underline{\phi}}^-,{\underline{\mu}}^-)|\tau_{k_1}(pt)\ldots\tau_{k_n}(pt)|({\underline{\phi}}^+,{\underline{\mu}}^+)\rangle_{g}^{\floor}$$
to be the weighted count of floor diagrams $D$ for $\mathbb{F}_k$ of degree $({\underline{\phi}},{\underline{\mu}})$ and genus $g$, with $n$ vertices with $\psi$-powers $k_1,\ldots,k_n$, such that $a$ equals the sum of all sizes of vertices, $a=\sum_{V=1}^n s_V$. 

Each floor diagram is counted with multiplicity $$\mult(D)=
\prod_{e \in C.E.} w(e)\cdot \prod_V \mult(V),$$
 where the first product is over the set $C.E.$ of  compact edges and $w(e)$ denotes their expansion factors; the second product ranges over all vertices $V$ and $\mult(V)$ denotes their multiplicities as in  Definition \ref{def-floormult}.
\end{definition}

\subsection{Motivation and relation to other work}\label{workflo}

For readers who are familiar with floor diagrams and their relation to tropical curves in $\mathbb{R}^2$, we insert a section discussing the special aspects of the definition we use here, and their relation to common definitions in the literature.

Floor diagrams were introduced for counts of curves in $\PP^2$ by Brugall\'e{}-Mikhalkin \cite{BM:Pn}, and further investigated by Fomin-Mikhalkin \cite{FM09}, leading to new results about node polynomials. The results were generalized to other toric surfaces, including Hirzebruch surfaces, in~\cite{AB13}.

The main observation is that by picking {\it horizontally stretched point conditions}, the images of tropical stable maps contributing to a Gromov--Witten invariant become floor decomposed: this means that the dual subdivision of the Newton polygon is {\it sliced} (i.e. a refinement of a subdivision  of the trapezoid by parallel vertical lines --- see Figure \ref{fig-slice}).
Floor diagrams are then obtained by shrinking each \textbf{floor} (i.e.\ a part of the plane tropical curve which is dual to a (Minkowski summand of a) slice in the Newton polygon)  to a white vertex. Each floor contains precisely one marked point. Further marked points lie on horizontal edges which connect floors, the so-called {\it elevators}\footnote{The bizarre nomenclature makes intuitive sense if everything is rotated by $90^\circ$.}, and are represented with black vertices. Fixed horizontal ends  are given a  (double circled) vertex, while other horizontal ends are shrunk  so that the diagram has no unbounded edges.
\begin{figure}[tb]
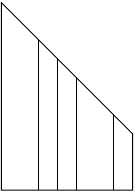
\caption{The subdivision of a floor decomposed tropical curve refines the sliced Newton polygon. Each strip corresponds to a floor, and the integral width of the strip is called the size of the floor.}\label{fig-slice}
\end{figure}

\begin{example}
In Figure \ref{fig-floors1}  we revisit the tropical stable map observed in the first part of Example \ref{ex-tropstablemap}. The floors are circled by dashed lines. On the right-hand side we have the corresponding floor diagram. Following the convention in \cite{BGM10}, fixed ends terminate with a double circle, and other ends are contracted to the corresponding black vertex.
\end{example}

\begin{figure}[tb]
\begin{tabular}{cc}
     \scalebox{0.6}{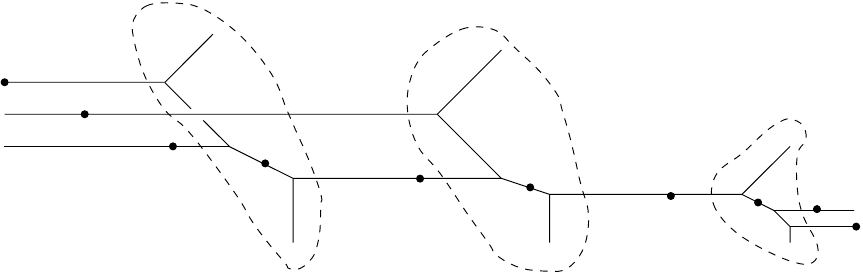}  
& %
     \scalebox{0.8}{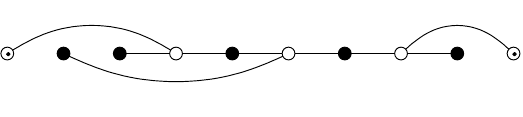}  

\end{tabular}
\caption{The floors in the tropical stable map  of Example \ref{ex-tropstablemap}(1), and the corresponding floor diagram.}\label{fig-floors1}
\end{figure}
 
 For rational stationary descendant Gromov--Witten invariants, the floor diagram technique was studied by Block, Gathmann and the third author \cite{BGM10} (the diagrams are called \textbf{$\psi$-floor diagrams}). There are two main differences with respect to the primary case:
 
 \begin{itemize}
 \item descendant insertions force us to consider floor decomposed curves with floors of size larger than one. The size of a floor thus becomes part of the data of a floor diagram: each floor vertex comes with two numbers, the $\psi$-power $k_i$ of the corresponding marked point, and  the size of the floor;
 \item marked points may now be supported at a vertex of the tropical stable map, and horizontal edges incident to such a vertex are fixed by the point condition. 
This condition is encoded by thickening the corresponding half-edges in the floor diagram.  
 \end{itemize}
 
\begin{example} 
Figure \ref{fig-floors2} illustrates the second part of Example \ref{ex-tropstablemap}. Some half-edges are thickened, indicating that the corresponding edge in the tropical curve leading to this diagram is adjacent to the marked point in the floor. 
\end{example}

\begin{figure}
\begin{tabular}{cc}
     \scalebox{0.6}{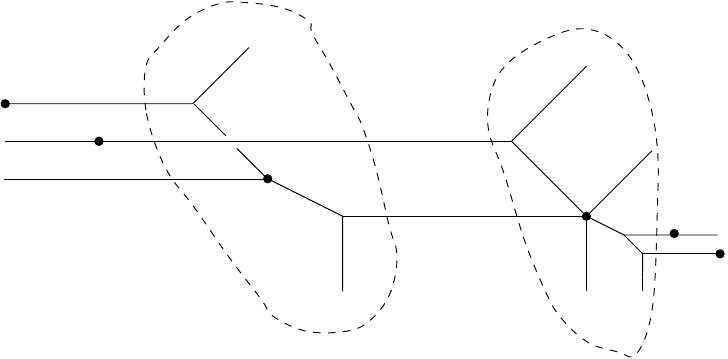}  
& %
     \scalebox{0.8}{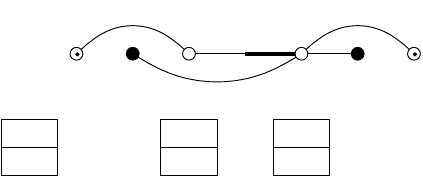}  

\end{tabular}
\caption{The floors in the tropical stable map  of Example \ref{ex-tropstablemap}(2), and the corresponding floor diagram. The numbers below the white vertices indicate the $\psi$-power $k_i$ of the marked point in the floor, and the size $s_i$ of the floor.}\label{fig-floors2}
\end{figure}

The language  in \cite{BGM10} was modeled after the work by Fomin--Mikhalkin \cite{FM09}, which was motivated by a computational approach aiming at new results about node polynomials. 
Our current motivation to study floor diagrams comes from their connections to degeneration techniques and Fock space formalisms to enumerative geometry.
 Hence our definitions introduce the following modifications with respect to \cite{FM09,BGM10}:
\begin{enumerate} 
\item The distinction between floors and marked points on elevators is not needed anymore (to the contrary, it only complicates the combinatorics and clouds the connection to the Fock space). We do away with bi-colored vertices by considering marked points on elevators as floors of size zero. The adjacent half-edges have to be thickened, since they are adjacent to the marked point. 
\item We  thicken ends that correspond to a tangency at a non-fixed point,  have unthickened ends for tangency to a fixed point (rather than marking the end with a double circle)  and remove the vertices at the end of these edges. 
\item We draw all elevator edges adjacent to marked points, as that allows us to record the complete tangency data for the invariant we are trying to compute (in the convention of \cite{FM09,BGM10}, obvious continuations of edges in the tropical curve are dropped in the floor diagram). 
\end{enumerate}
As an example,  the floor diagram of Figure \ref{fig-floors2} becomes with our conventions the diagram in Figure \ref{fig-floordiagram2a}.


\subsection{Floor diagrams and  degeneration}
\label{claflo} As seen in the previous section, tropical curves naturally arise by counting curves in maximal degenerations, while the correspondence from relative descendants to floor diagrams follows from the simpler ``accordion'' degeneration. This was originally observed and discussed in \cite{Bexp, AB14}. We work in this section with Jun Li's degeneration formula for relative stable maps. We recall Li's theorem, stated in the specific geometric context that is of interest to us~\cite{Li02}. 

%

 \begin{theorem}
 \label{degfor}
 Let ${\mathscr X}_t$ be a flat family of surfaces such that the general fiber is a smooth Hirzebruch surface $\mathbb{F}_k$ and the central fiber is the union of  two surfaces $S_1\cup_{D} S_2$ both isomorphic to $\mathbb{F}_k$, meeting transversely along the divisor $D = E_{S_1} = B_{S_2}$. Fix a two part partition of the set $[n+m]$: without loss of generality we may choose $\{1, \ldots n\} \cup \{n+1, \ldots, n+m\}.$ Set  discrete invariants $g,n, k_1,  \dots, k_n, ({\underline{\phi}}, {\underline{\mu}})$   as in Notation \ref{not-delta}.
 
Then:
 \begin{align} \label{for-deg}\begin{split}
& \langle  ({{\underline{\phi}}^-}, {{\underline{\mu}}^-})| \tau_{k_1}(pt),\ldots, \tau_{k_{n+m}}(pt)|({{\underline{\phi}}^+}, {{\underline{\mu}}^+})\rangle^{\mathsf{rel},\bullet}_{g} = 
\sum \frac{\prod{\lambda_i \eta_j}}{|Aut({\underline{\lambda}})||Aut({\underline{\eta}})|} \cdot\\
& \langle ({{\underline{\phi}}^-}, {{\underline{\mu}}^-})|\tau_{k_1}(pt),\ldots, \tau_{k_{n}}(pt)|({\underline{\lambda}},{\underline{\eta}})\rangle^{\mathsf{rel},\bullet}_{ g_1}  \langle ({-\underline{\eta}},{-\underline{\lambda}})| \tau_{k_{n+1}}(pt),\ldots, \tau_{k_{n+m}}(pt)|({{\underline{\phi}}^+}, {{\underline{\mu}}^+})\rangle^{\mathsf{rel},\bullet}_{g_2},\end{split}
 \end{align}
 where ${\underline{\lambda}} = \lambda_1, \ldots, \lambda_r, {\underline{\eta}} = \eta_1, \ldots, \eta_s$ are an $r$-tuple and and an $s$-tuple of positive integers and the sum is over all discrete data $(g_1, g_2, ({\underline{\eta}},{\underline{\lambda}}))$ such that:
 \begin{enumerate}
		\item $ (({{\underline{\phi}}^-} , {\underline{\lambda}}), ({{\underline{\mu}}^-}, {\underline{\eta}}))$ (resp. $(( {-\underline{\eta}},{{\underline{\phi}}^+} ), ({ -\underline{\lambda}},{{\underline{\mu}}^+}))$) determines an effective curve class $a_1B_{S_1}+b_1F_{S_1}$ (resp. $a_2B_{S_2}+b_2F_{S_2}$) in $H_2(\mathbb{F}_k,\ZZ)$  with $a_1, a_2 \geq 0$, $a_1+a_2 = a$, $b_1 = a_2k+b$, $b_2=b$;
	\item $g = g_1 +g_2 + r+s -1$.
 \end{enumerate}
 
 \end{theorem}
 
Note that the superscript $\bullet$ refers to disconnected Gromov-Witten theory. 
 
 
 \begin{remark}
The following details are important in parsing Equation \eqref{for-deg}: 
\begin{enumerate}
\item The formula is organized as a sum over the gluing data $(\underline{\lambda}, \underline{\eta})$. Each term in the summand is  however weighted by a factor of $\frac{1}{|Aut({\underline{\lambda}})||Aut({\underline{\eta}})|}$, which corrects the overcounting coming from different labelings of points that give rise to the same gluing. More geometrically, one may think that in Equation \eqref{for-deg} the sum is over the distinct topological types of maps (where the points that get glued are unlabeled), and the multiplicity of each summand omits the above factor.
\item The switching of the roles of $\underline{\lambda}$  and $\underline{\eta}$ on the two sides of the product comes from the Kunneth decomposition of the class of the diagonal in $\mathbb{P}^1\cong D$.
\end{enumerate}
 \end{remark}
 
To realize the hypotheses of the theorem, one may start from a trivial family $\mathbb{F}_k\times \mathbb{A}^1 \to \mathbb{A}^1$ together with $n+m$ non intersecting sections $s_v$, the first $n$ staying away from $E$, the last $m$ meeting but not tangent to $E$ at $t=0$; one obtains ${\mathscr X}_t$ by blowing up $E\times \{0\}$ and considering the proper transforms of the sections. This construction may be iterated a finite number of times, and Theorem \ref{degfor} applies with the appropriate bookkeeping. This is what gives rise to the correspondence with the floor diagram count, as we make explicit in the next theorem. We have included the statement and proof for completeness. Essentially identical results may be found in~\cite{MR16,NS06}.

\begin{theorem}\label{thm-flooralg}

Fix a Hirzebruch surface $\mathbb{F}_k$ and discrete data as in Notation \ref{not-delta}. 
The descendant Gromov--Witten invariant coincides with the weighted count of floor diagrams from Definition \ref{def-floorcount}
\begin{equation}
\langle  ({{\underline{\phi}}^-}, {{\underline{\mu}}^-})| \tau_{k_1}(pt),\ldots, \tau_{k_{n}}(pt)|({{\underline{\phi}}^+}, {{\underline{\mu}}^+})\rangle^{\bullet, \floor}_{g}= \langle  ({{\underline{\phi}}^-}, {{\underline{\mu}}^-})| \tau_{k_1}(pt),\ldots, \tau_{k_{n}}(pt)|({{\underline{\phi}}^+}, {{\underline{\mu}}^+})\rangle^{\mathsf{rel},\bullet}_{g}.
 \end{equation}
\end{theorem}

\begin{proof}
As with many proofs based on iterated applications of the degeneration formula, a completely explicit and accurate bookkeeping would be extremely cumbersome and cloud the actual simplicity of the argument. We choose therefore to carefully outline the construction, and omit the bookkeeping.

Iterate the construction from the previous paragraph $n-1$ times, each time separating exactly one section from all others. In the end one obtains a family ${\mathscr X}_t$  such that the general fiber is a smooth Hirzebruch surface $\mathbb{F}_k$ and the central fiber is the union of  $n$ surfaces $X_0 = S_1\cup_{D_1} S_2\cup_{D_2}\ldots \cup_{D_{n-1}} S_n$, where all surfaces $S_i$ are isomorphic to $\mathbb{F}_k$, and $S_i$ and $S_{i+1}$ meet transversely along the divisor $D_i = E_{S_i} = B_{S_{i+1}}$. For $V = 1, \ldots, n$ the section $s_V$ is obtained as the proper transform of the original section. 

Applying the appropriately iterated version of Theorem \ref{degfor}, the stationary descendant invariant is expressed as a sum over the topological types of maps from nodal curves to the central fiber, weighted by the appropriate (disconnected) relative Gromov--Witten invariants. Since each $S_i$  contains exactly one marked point, the disconnected maps to $S_i$ have one connected component hosting the marked point; by dimension reasons, the other components consist of rational curves mapping with degree   $dF$ (multiple  of the class of a fiber), and in fact mapping as a $d$-fold cover of a fiber, fully ramified at the points of contact with $D_{i-1}$ and $D_{i}$. Further, the relative conditions
at the boundary must have one fixed point on one side, and a moving point on the other. The contribution of any such component to the disconnected invariant is $1/d$.

For every summand in the degeneration formula, consider the dual graph of the source curve, label each edge with the ramification order of the corresponding point, and thicken half edges corresponding to moving boundary point conditions. For every two-valent vertex adjacent to two flags of opposite thickening, contract the vertex and the two neighboring flags. We claim (and leave the verification to the patient reader) that the object thus obtained is a floor diagram for the stationary descendant invariant we are trying to compute, and further  that this construction establishes a bijection between the summands in the degeneration formula and the floor diagrams described in Definition \ref{def-floorcount}.

The proof is concluded by showing that each floor diagram is counted with the same multiplicity. The degeneration formula assigns the same vertex and compact edge multiplicities to the dual graphs of maps as the floor diagram enumerative count. The proof is then concluded by noticing that the operation of removing a two-valent vertex (which contributes with multiplicity $1/d$) and its two adjacent flags does not alter the multiplicity of the graph:  for each such vertex removed we lose a compact edge of weight $d$, which contributes a factor of $d$ to the multiplicity of the graph.
\end{proof}

Since the proof of the correspondence is  based on a bijection between dual graphs of maps and floor diagrams that preserves connectedness, one immediately obtains the following corollary. 

\begin{corollary}
The version of Theorem \ref{thm-flooralg} for connected invariants also holds:
\begin{equation}
\langle  ({{\underline{\phi}}^-}, {{\underline{\mu}}^-})| \tau_{k_1}(pt),\ldots, \tau_{k_{n}}(pt)|({{\underline{\phi}}^+}, {{\underline{\mu}}^+})\rangle^{ \floor}_{g}= \langle  ({{\underline{\phi}}^-}, {{\underline{\mu}}^-})| \tau_{k_1}(pt),\ldots, \tau_{k_{n}}(pt)|({{\underline{\phi}}^+}, {{\underline{\mu}}^+})\rangle^{\mathsf{rel}}_{g}.
 \end{equation}
\end{corollary}

\section{Floor diagrams via the tropical theory}\label{troflo}

In this section, we examine the effect of {horizontally stretched} constraints on tropical maps. Tropical maps meeting such constraints {are floor decomposed, and} better behaved than for general conditions. This allows us to prove a degeneration formula for the logarithmic descendants, expressing them as a product over vertices. We use this to present a direct weighted bijection between counts of tropical stable maps and floor diagrams. 

If we consider the enumerative problem with $k_i=0$ for all $i$, then the tropical descendant invariant considered above is nothing but a count of tropical plane curves satisfying point conditions. For such a count, it is well-known, see for instance~\cite{Mi03}, that all tropical stable maps $(\Gamma,f)$ that contribute with nonzero multiplicity have $g$ cycles which are visible in the image $f(\Gamma)$ and have only trivalent vertices. It then follows that their spaces of deformations have the expected dimension, because the $g$ visible cycles impose $2g$ linearly independent conditions in the orthant parametrizing all lengths on bounded edges. Hence superabundancy is not present for tropical plane curve counts. 

The primary tool in this section is a specialization of the tropical points into special configuration. 

\begin{definition}
A set $p_1,\ldots, p_n\in \RR^2$, with $p_i = (x_i,y_i)$ is said to be in \textbf{horizontally stretched position} if 
\begin{itemize}
\item The $x$-coordinates increase, i.e. $x_i<x_{i+1}$ for all $i$.
\item The $y$-coordinates decrease, i.e. $y_i<y_{i+1}$ for all $i$.
\item The $x$ coordinates are much larger than the $y$-coordinates,
\[
\min_{i\neq j}|x_i-x_j|\gg \max_{i\neq j}|y_i-y_j|.
\]
\end{itemize} 
\end{definition}

As observed by Brugall\'e, Fomin, and Mikhalkin, the choice of distinguished stretching direction determines a distinguished decomposition of a tropical map.

\begin{definition}
Let $\Gamma\to \RR^2$ be a tropical stable map. An \textbf{elevator edge} is an edge in $\Gamma$ whose image has edge direction parallel to $(1,0)$. A \textbf{floor} of a tropical map $\Gamma\to \RR^2$ is a connected component in $\Gamma$ of the complement of all elevator edges. 
\end{definition} 

Note that if a descendant lies \textit{on} an elevator edge, the corresponding vertex is a floor supporting the marked point. The following proposition shows the usefulness of the notions above. 

\begin{proposition}
Let $\Gamma\to\RR^2$ be a rigid tropical stable map meeting horizontally stretched stationary descendant constraints. Each floor of $\Gamma$ meets and is fixed by exactly one stationary descendant constraint. 
\end{proposition}

\begin{proof}
The proof follows from identical arguments as in~\cite{FM09}, Lemma 3.6: if there was no marked point in a floor, we could vary the vertical position of the floor and the curve would not be rigid. The vertical position is fixed as soon as one marked point is present. 
\end{proof}

A tropical stable map meeting the horizontally stretched descendant constraints will be called \textbf{floor decomposed}.

\subsection{Excluding superabundant maps} When descendant insertions are allowed, even if all cycles are visible in the image, they do not need to impose linearly independent conditions.
 The existence of superabundant tropical stable maps satisfying the conditions implies the existence of rigid tropical stable maps with ''additional overvalency'', e.g.\ a $4$-valent vertex which is not adjacent to a marked point (see Example 3.10 in \cite{GM051}). In what follows, we restrict such behavior for the case of floor-decomposed tropical stable maps. 
 
 Let $(\Gamma,f)$ be a tropical stable map and $L\subset \Gamma$ be a cycle. We say that \textbf{the cycle $L$ is visible} if for any open neighborhood $U_L$ of $L$ in $\Gamma$, its image $F(U_L)$ is not contained an affine line in $\mathbb R^2$. Note that when $\Gamma$ is trivalent, the balancing condition guarantees that $L$ is contained in an affine line if and only if some neighborhood of $L$ is contained in an affine line, so $L$ is indeed ``visible'' in the image. The definition is made in order to appropriate capture the degenerations.

\begin{proposition}\label{lem-cyclesvisible}
If $(\Gamma,f)$ is a floor decomposed rigid tropical stable map contributing to the Gromov--Witten invariant with nonzero multiplicity, then the following conditions hold.
\begin{enumerate}[(A)]
\item No cycle of $\Gamma$ is contracted to a point.
\item All cycles of $\Gamma$ are visible in $f(\Gamma)$.
\item If there exists vertices $V$ and $W$ of $\Gamma$ with two or more edges between them, then these edges are mapped horizontally by $f$, i.e. the edge is contained in a line parallel to the $x$-axis. 
\item In the situation of (C), either $V$ or $W$ carries a marked point of $\Gamma$. 
\end{enumerate}
\end{proposition}

\begin{proof} Since we assume that $(\Gamma,f)$ contributes with non-zero multiplicity, it has to be rigid.
If a cycle of $\Gamma$ was contracted to a point, then $(\Gamma,f)$ would not be rigid because the lengths of edges of the contracted cycle can be varied without changing the image $f(\Gamma)$. We could vary $(\Gamma,f)$ in an at least one-dimensional family still meeting the point and $y$-coordinate conditions. This establishes the first claim.

Assume an open neighborhood of a cycle of $\Gamma$ is mapped to a line segment $S$ in $\mathbb R^2$. Let $V$ be a vertex of the cycle mapping to an endpoint of $S$. The genericity of incidence conditions guarantees that only one side of this flat cycle bears a marking, and we take $V$ to be the vertex without a marking. Since we have already ruled out  contracted cycles, the segment $S$ has two distinct endpoints. Such a map $(\Gamma,f)$ does not contribute to the logarithmic invariant. Indeed, $(\Gamma,f)$ is once again non-rigid. Since an open neighborhood of a circuit is contracted to $S$, the length of the image of $L$ itself may be varied without changing the image (see  Figure \ref{fig-move}), implying non-rigidity. This establishes the second claim. 

Let $V$ and $W$ be edges of $\Gamma$ with multiple (i.e. parallel) edges between them. If $f$ did not map the edges between $V$ and $W$ horizontally, $V$ and $W$ would be in the same floor. Assume without loss of generality that $V$ is higher than and to the right of $W$. Recall, we impose that the ends in the positive $y$-direction all have slope $1$. Since there are at least $2$ parallel edges, there are at least $2$ ends in the floor containing $V$ and $W$. The horizontal distance between these two parallel ends, however, cannot be fixed by the point constraints, so we conclude that such a map $(\Gamma,f)$ cannot be rigid, giving us the third claim. 

For the final claim, we consider the situation of a rigid curve as in $(C)$, with parallel multiple horizontally mapped edges between $V$ and $W$, where neither carries a marked point. By definition $V$ and $W$ are in different floors, but neither $V$ nor $W$ has its position fixed by a marked point. We may therefore vary the vertical coordinate of edge in the image curve between $f(V)$ and $f(W)$ while still meeting all of the descendant constraints (consider Figure~\ref{fig-multzero} dropping the marked point on the left vertex of the double edge). Such a curve could therefore never be rigid, and no such tropical curve contributes to the Gromov--Witten invariant. 

\end{proof}
\begin{figure}[bt]
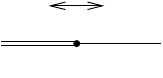
\caption{Tropical maps with a trivalent vertex at the end of a hidden cycle are not rigid.
}\label{fig-move}
\end{figure}

\begin{lemma}\label{lem-cyclesindependent}
Let $(\Gamma,f)$ be a rigid floor decomposed tropical stable map, satisfying horizontally stretched conditions. Assume that $g'$ cycles, the images of their edges span $\mathbb{R}^2$, then these cycles impose $2g'$ linearly independent conditions on the lengths of the bounded edges of $\Gamma$.
\end{lemma}

\begin{proof}
The conditions imposed by the cycles are given as a $2g'\times b$ matrix, where $b$ denotes the number of bounded edges. We show that this matrix has rank $2g'$.
Since $(\Gamma,f)$ is rigid, a cycle cannot be contained in a floor, as each floor is fixed by a single point condition. Each of the $g'$ cycles thus have to connect at least two floors, and therefore involve at least two elevator edges.
If there are cycles like in Proposition \ref{lem-cyclesvisible} (C), i.e.\ flat cycles consisting of multiple elevator edges, we can choose one of the multiple edges for each, and our basis of  $H_1(\Gamma)$, in such a way that each of the $g'$ cycles contains at most our preferred of the multiple edges. 

Each cycle contributes two rows to our matrix. We first consider the $x$-coordinate rows. We pick a cycle $C_1$ and an elevator edge $e_1$ in it. We then adapt our basis of $H_1(\Gamma)$ such that no other cycle contains $e_1$. We iterating this procedure. The first $g'$ rows of our matrix are the $x$-coordinate rows of the cycles in the thus chosen order, and the first $g'$ columns are the elevator edges we picked for each.
Then our matrix has a block form, with a $g'\times g'$ triangular block on the top left. 

We next discuss the rows with the $y$-coordinates of our cycles. 
Since $(\Gamma,f)$ is rigid, every floor edge has a direction vector with $y$-coordinate $\pm 1$: as in the proof of Proposition \ref{lem-cyclesvisible} (C) we would otherwise have two parallel ends whose horizontal position could not be fixed by the point conditions. Thus every entry in our matrix belonging to a $y$-coordinate row of a cycle and the column of a floor edge is $\pm 1$, or $0$ if the edge does not belong to the cycle.

Notice that the set of all floor edges of a cycle determines the cycle: this is true since for our choice of point conditions, we can only have elevators with the same $y$-coordinate if they are adjacent to the same vertex.

Thus we can now order our cycles (and, accordingly, the rows of the matrix containing their $y$-coordinates) in a way refining the partial order given by containement of the set of floor edges in a cycle.
The $g'\times b-g'$ block on the bottom right does not need to have a triangular form, since it is possible that floor edges belong to several cycles. We start with the top row, belonging to the minimal cycle. We order the columns such that the first entry in the $g'\times b-g'$ is nonzero. We use it to reduce all nonzero entries below it in the lower rows. Since the lower rows correspond to larger cycles in our order, we cannot have canceled all nonzero entries in these rows in this way, this could only happen if the larger cycles had exactly the same floor edges. Thus, we can take the next Pivot column in the second row, and use it again to reduce the nonzero entries below it. Again, there will be nonzero entries left in the third row and we continue with the next Pivot. Continuing like this, we reduce our matrix to a row echelon form which has full rank. It follows that the $2g'$ cycles give independent conditions on the edge lengths.

\end{proof}
For an example, see Figure \ref{fig-cycles}. The elevator edges appearing in cycles are $e_1$, $e_3$, $e_5$ and $e_6$. We let $C_1$ be any cycle containing $e_1$, and $C_2$ the unique cycle not containing $e_1$. The edge $e_2$ is part of $C_1$, but not of $C_2$. The two equations for $C_1$ and $C_2$ restricted to the four edges $e_1,\ldots,e_4$ produce an upper triangular matrix immediatly in this example, without any reduction being necessary. 

\begin{figure}[tb]
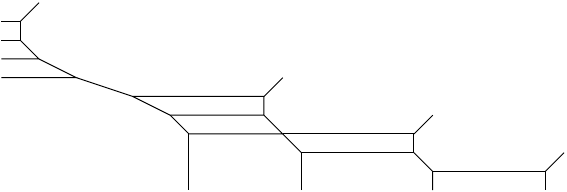
\caption{The visible cycles in a floor decomposed tropical stable map pose independent conditions.}\label{fig-cycles}
\end{figure}

\begin{figure}
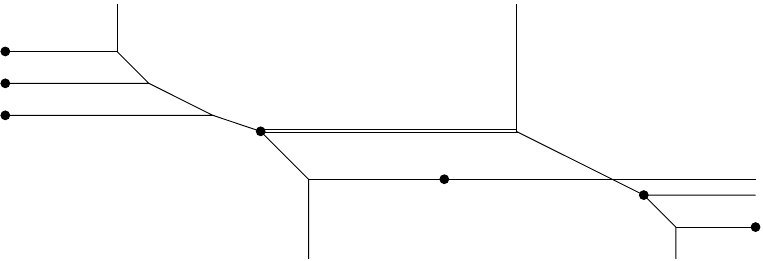
\caption{A rigid tropical stable map contributing to the Gromov--Witten invariant. Note that the parallel edges must be horizontal, an open neighborhood of the cycle spans $\RR^2$, and no cycles are contracted to a point.}\label{fig-multzero}
\end{figure}

Combining Lemma~\ref{lem-cyclesindependent} and Proposition~\ref{lem-cyclesvisible}, we  deduce the following non-trivial fact:
\begin{corollary}\label{cor-notsuperabundant}
For horizontally stretched point conditions leading to floor-decomposed tropical stable maps, any $(\Gamma,f)$ that contributes with a non-zero multiplicity to a tropical descendant Gromov--Witten invariant is non-superabundant.
\end{corollary}
\begin{proof}
From Proposition \ref{lem-cyclesvisible} we can conclude that all cycles of $\Gamma$ are visible in the image $f(\Gamma)$. From Lemma \ref{lem-cyclesindependent} we can conclude that those whose edges span $\mathbb{R}^2$ each impose two conditions, which are independent. Each flat cycle as in Proposition \ref{lem-cyclesvisible} (C) gives one condition, and the set of all conditions of the visible cycle is also independent. Also, each flat cycle imposes the existence of higher-valent vertices. It follows that the space of deformations of $(\Gamma,f)$ is of the expected dimension, and hence $(\Gamma,f)$ is not superabundant.
\end{proof}

\begin{remark}
For tropical stable maps to $\mathbb{R}^n$ with $n\geq 3$, there is no analogous statement known, i.e.\ it is not known whether there is a configuration of points such that all tropical stable maps (of non-zero multiplicity) satisfying the conditions are not superabundant, or even if there is a configuration of points forcing all cycles to be visible, which is a much weaker condition. It would be interesting to study the effect of floor stretched conditions in higher dimensions, particularly in light of~\cite[Theorem~1.1(1)]{MR16} and its potential applicability.
\end{remark}

%

\begin{lemma}\label{lem-3valent}
Let $(\Gamma,f)$ be a floor-decomposed tropical stable map contributing to a tropical descendant Gromov--Witten invariant with non-zero multiplicity. Then every vertex of $\Gamma$ which is not adjacent to a marked end or a flat cycle is trivalent of genus $0$.
\end{lemma}
\begin{proof}

Let $(\Gamma,f)$ be a rigid floor-decomposed tropical stable map.
Assume the marked points with $\psi$-conditions $k_1,\ldots,k_n$ are at vertices of genus $g_1,\ldots,g_n$, and accordingly, of valence $k_i+3-g_i$.


By Corollary \ref{cor-notsuperabundant}, $(\Gamma,f)$ is not superabundant. The number of edges in the graph $\Gamma$ is 
\[
n+n_1+n_2+2a-3+3(g-g_1-\ldots-g_n)-\sum_V (\val(V)-3),
\]
which follows from an Euler characteristic computation. Assume that of the $g-g_1-\ldots-g_n$ visible cycles, $g'$ are flat cycles as in Proposition \ref{lem-cyclesvisible} (C) and for $g''$, their edges span $\mathbb{R}^2$. Note that two vertices connected by $m$ multiple edges give $m-1$ flat cycles.

The space of deformations of $(\Gamma,f)$ has dimension: 
\begin{align*}  &  2+ \#\{edges \}  - 2\cdot g'' - g'  \\=&n+n_1+n_2+2a-1+3(g-g_1-\ldots-g_n)-\sum_V (\val(V)-3) -2(g''+g)'+g' \\=&n+n_1+n_2+2a-1+(g-g_1-\ldots-g_n)-\sum_V (\val(V)-3) +g' \\=&n+n_1+n+\sum_i (k_i-g_i)+g'-\sum_V (\val(V)-3)\end{align*} by the requirement on the conditions.
Since $\sum_V (\val (V)-3)=\sum_i(k_i-g_i)+g'+\sum_{V'}(\val(V')-3)$ (where now the sum goes over all vertices $V'$ which are not adjacent to one of the marked ends $i$ or to flat cycles) by the valency conditions, and since the $y$-coordinates of $n_1$ ends are fixed and $n$ generic point conditions are satisfied, the dimension has to be at least $n_1+2n$, which can only be satisfied if any vertex besides the ones adjacent to the marked ends, is trivalent and of genus $0$.
\end{proof}

\subsection{Deformations of rigid floor decomposed maps} Let $(\Gamma,f)$ be a floor decomposed, rigid tropical stable map contributing to a stationary descendant invariant. We wish to \rc{express} the virtual multiplicity $m_{(\Gamma,f)}$ attached to such a floor decomposed curve in terms of its vertices. Specifically, each vertex $V$ of $\Gamma$ gives rise to discrete data for a logarithmic stable map. Distributing the stationary conditions, as well as boundary conditions along outgoing edges at $V$, we obtain a candidate for local multiplicities which should comprise $m_{(\Gamma,f)}$. In order to carry out this strategy, we need to understand the non-rigid tropical maps nearby $(\Gamma,f)$. 

\begin{proposition}\label{prop: non-rigid-defs}
Let $(\Gamma,f)$ be a rigid tropical map as above and let $(\Gamma',f')$ be a deformation of it, meeting the same stationary descendant constraints. Then, the image of $f$ coincides with the image of $f'$. 
\end{proposition}

\begin{proof}
As a consequence of Proposition~\ref{lem-cyclesvisible} in the previous section, any rigid curve $(\Gamma,f)$ must be trivalent and of genus $0$ away from its marked ends and flat horizontal cycles. Moreover, if $W$ is a vertex of $\Gamma$ without a marking and with high valency, then the image curve at $f(W)$ is still trivalent. We will show that locally near a vertex of $\Gamma$, no deformation that changes the image of $f$ can meet the stationary descendant conditions. 

In the first instance, examine a vertex $W$ without a marking but with high valency as above. Then $W$ is connected by a flat horizontal cycle to a vertex $V$ that does carry a descendant. 

Assume that $\Gamma$ has a single vertex $V$ of genus $g$, supporting a marked point, and having valency $r$. Let $\ell$ be the power of the descendant attached to the marking. Then we have the equality $r-3+g = \ell$. 

The image of $f$ is dual to a Newton polygon $\Delta_f$ with at most $r$ sides and at least $g$ interior lattice points. The deformations of $(\Gamma,f)$ that change the image correspond to a subdivision $\Delta_f'$ of the Newton polygon $\Delta_f$. Recall that the $2$-dimensional polygons contained in the subdivision $\Delta_f'$ are dual to the vertices in the deformed tropical (image) curve. Moreover, the genus of each such vertex is equal to the number of interior lattice points. In order to meet the descendant condition,  we must produce a subdivision $\Delta_f'$ such that for some vertex $V'$ dual to a polygon in the subdivision, we achieve the equality
\[
\mathrm{val}(V')-3+g(V') = \ell.
\]

The vertices of the polygons in $\Delta_f'$ can include the interior lattice points of $\Delta_f$. We first deal with the case when there are no parallelograms in the subdivision. Then any interior point that is used produces a visible cycle in the image. Assume $g_1$ interior points are used in constructing $\Delta_f'$. In the resulting tropical curve, any vertex $V'$ can have genus at most $g-g_1$. 
A polygon in the subdivision can have at most $g_1+1$ edges that are not edges of $\Delta_f'$, i.e. ``new edges'' that appear in the subdivision. Furthermore, any polygon in the subdivision can have at most $r-2$ boundary edges of $\Delta_f$. Thus, we see that
\[
(r-2+g_1+1)-3+(g-g_1)<\ell
\] 
and the deformation cannot meet the descendant condition. 

Now assume that there is a parallelogram in $\Delta_f'$. Here, it is not necessarily the case that the genus of the deformed curve is smaller. However, the parallelogram is dual to two edges of the tropical curve crossing, whose dual edges must be adjacent to different polygons in the subdivision. Thus, a vertex can have at most $r-3$ edges dual to boundary edges of $\Delta_f$, and we again have the inequality
\[
(r-3+g_1+1)-3+(g-g_1+1)<\ell.
\]
Thus, we see that any parallelogram in $\Delta_f'$ reduces the valency, while polytopes in $\Delta_f'$ other than parallelograms reduce the genus. In all cases, it is impossible that the deformed curve continues to meet the descendant condition. Thus, the only deformations of $(\Gamma,f)$ must leave the image unchanged, and the proposition follows. 
\end{proof}


\subsection{Tweaking the logarithmic moduli space} The consequence of Proposition~\ref{prop: non-rigid-defs} is that \textit{all} tropical stable maps -- not just rigid ones -- that contribute to a stationary descendant Gromov--Witten invariant have the same image as a rigid map. Thus, there are only finitely many images of tropical curves that we need consider in the enumerative problem. This allows us to use an elegant idea developed by Gross--Pandharipande--Siebert~\cite{GPS}, building on earlier work of Nishinou--Siebert~\cite{NS06}, to relate the logarithmic stationary descendants of the Hirzebruch surface $\mathbb F_k$ to the stationary descendants of an \textit{open} geometry obtained by first degenerating $\mathbb F_k$ to accommodate all tropical curves, and then deleting its codimension $2$ strata. The resulting moduli space will not be complete, but retains sufficient properness to support the stationary descendants because of the results of the previous section. This in turn will give us access to the degeneration formula. 

Fix the numerical data defining a tropical stationary descendant Gromov--Witten invariant. Choose floor decomposed points, and let $(\Gamma_1,f_1),\ldots, (\Gamma_s,f_s)$ be the rigid tropical curves meeting these conditions. Let $\mathscr P$ be a polyhedral decomposition of $\RR^2$ such that for each $i$, the image
\[
f_i: \Gamma_i\to \RR^2
\]
is contained in the $1$-skeleton of $\mathscr P$. 

\begin{proposition}
Let $f':\Gamma'\to \RR^2$ be any tropical stable map meeting the given stationary descendant constraints as above. Then $f'$ factors through the one skeleton  of $\mathscr P$. 
\end{proposition}

\begin{proof}
This is a restatement of Proposition~\ref{prop: non-rigid-defs} in the previous section. 
\end{proof}

Let $\mathscr X$ be the special fiber of the toric degeneration associated to $\mathscr P$, and let $\mathscr X^\circ$ denote the complement of the codimension $2$ strata in the degeneration. We consider $\mathscr X^\circ$ as a logarithmic scheme over $\spec(\NN\to \CC)$ equipped with its divisorial logarithmic structure. 

Let $(\Gamma,f)$ be rigid and let $\overline{M}_{(\Gamma,f)}(\mathscr X^\circ)$ denote the space of logarithmic stable maps to $\mathscr X^\circ$ equip-ped with a marking by $(\Gamma,f)$ in the sense of Section~\ref{sec: virtual-multiplicity}. Choose logarithmic lifts $\underline p^\dagger =(p_1^\dagger,\ldots, p_n^\dagger)$ of the points $(p_1,\ldots, p_n)\in (\RR^2)^n$. Let $\overline{M}_{(\Gamma,f)}(\mathscr X^\circ,\underline p^\dagger)$ be the moduli space of logarithmic stable maps passing through these points.

 This space admits a virtual fundamental class $[\overline{M}_{(\Gamma,f)}(\mathscr X^\circ,\underline p^\dagger)]^{\mathsf{vir}}$ in Borel--Moore homology. Moreover, the stationary descendant logarithmic invariant is rationally equivalent to a class supported on $\overline{M}_{(\Gamma,f)}(\mathscr X^\circ,\underline p^\dagger)$. 

We view the space $\mathscr X^\circ$ as a logarithmically smooth scheme with a rank $1$ logarithmic structure, since the higher rank loci are the (at least) threefold intersections of components which have been removed. 

The moduli space of maps to $\mathscr X^\circ$ includes into the moduli space of maps to $\mathscr X$ by composition. By the arguments in previous section, the difference
\[
\left(\prod_{j=1}^n\psi_j^{k_j}\cap [\overline{M}_{(\Gamma,f)}(\mathscr X^\circ,\underline p^\dagger)]^{\mathsf{vir}}\right)
\]
is equal to the stationary descendant invariant in question. Indeed, any logarithmic map meeting the stationary descendant constraints has to have a tropicalization that factors through the one skeleton of $\mathscr P$, and thus lie in the open target. Thus, the invariant coming from the expanded moduli space has a well-defined degree. 

%

\subsection{Vertex multiplicities} Since the degeneration $\mathscr X^\circ$ has no triple points, we may now appeal to the degeneration formula for smooth pair geometries. Specifically, the virtual class of $\overline{M}^{\mathsf{}}_{(\Gamma,f)}(\mathscr X^\circ,\underline p^\dagger)$ satisfies a degeneration formula analogous to the one used in the previous section, and we use the one proved by Kim--Lho--Ruddat~\cite{KLR}. This allows us to write the virtual multiplicity $m_{(\Gamma,f)}$ in terms of the vertices of $(\Gamma,f)$. 

Let $(\Gamma,f)$ be a rigid tropical stable map as above, contributing to a stationary descendant Gromov--Witten invariant. Orient the edges of $\Gamma$ minus the marked ends in each component towards the unique  non-fixed end.

\noindent
{\bf From oriented edges to boundary incidence conditions.} Locally around each vertex $V$ of $\Gamma$, the directions of the adjacent flags define a Newton fan $\delta_V$. We let $\delta_\phi$ be the subset given by all entries of $\delta_V$ corresponding to edges which are oriented towards $V$, and $\delta_\mu$ consist of the vectors in $\delta_V$ oriented away from $V$. \\

Let $\overline M_V$ is the moduli space of maps to the open surface $X_V^\circ$ determined by the local picture near $V$, such that if an edge is incoming, we consider maps that pass through a pre-determined point of the corresponding boundary curve.  

More precisely, by unpacking the degeneration formula in~\cite{KLR}, we see that at a vertex $V$ there is a corresponding moduli space of maps $f:C_V\to X_V$, where the surface $X_V$ is determined by the Newton fan $\delta_V$ and the genus, marked points, and contact orders are given by the star of $V$ in $(\Gamma,f)$. We consider the open toric surface $X_V^\circ$ obtained by deleting the torus fixed point. Moreover, let $\overline M_V$ be the moduli space of logarithmic maps to $X_V^\circ$, with the boundary incidences specified by the orientation as follows. If an edge is oriented towards $V$, we consider maps passing through a fixed point of the corresponding boundary curve. 

This moduli space of maps admits a virtual fundamental class $[\overline M_V]^{\mathsf{vir}}$. 

\begin{definition}\label{def: local-vertex-multiplicity}
Define the \textbf{local vertex multiplicity at $V$} to be 
$$\mult_V(\Gamma,f)=\langle \tau_{k_i}(pt)\rangle_{\delta_\phi\cup \delta_\mu,g_V} := \int_{[\overline M_V]^{\mathsf{vir}}} \psi_i^{k_i} ev^\star([pt]). 
$$ 
if the marked end $i$ is adjacent to $V$ and
\[\mult_V(\Gamma,f)=\langle \rangle_{\delta_\phi\cup \delta_\mu ,g_V}:=\mathsf{deg}([\overline M_V]^{\mathsf{vir}})
\]
otherwise.
Here $g_V$ denotes the genus of $\Gamma$ at $V$.
\end{definition}
Since we require the marked ends to meet distinct points, there cannot be more than one end adjacent to a vertex $V$.

Note that the arguments in the previous section guarantee that the possible degenerations of the local map near $V$ that satisfy the valency and incidence conditions are contained in the a priori non-compact space of maps $\overline M_V$. 

Given a an edge $e$ of a rigid tropical curve $(\Gamma,f)$, we let $w(e)$ denote the expansion factor along the edge.


\begin{remark}\label{dimcount}
The only possibly non-vanishing local vertex multiplicities happen when the virtual dimension of the moduli space of logarithmic stable maps equals $0$ in the case of an unmarked vertex, and $k_i+2$ for a vertex adjacent to the $i$-th mark. Let $V$ denote a vertex whose star gives the Newton fan $\delta$.  Let $\delta_\phi\cup \delta_\mu = \delta$ be an arbitrary two-part partition of $\delta$, and let $M_V$ the moduli space of logarithmic stable maps identified by this data. The virtual dimension  is:
\begin{equation}
\mathsf{virdim}(M_V) =  g - 1 + \val(v) - \ell(\phi)
\end{equation}

If $V$ is an unmarked vertex, using  $val(V) = \ell(\phi) +\ell(\mu)$, it follows that for  the virtual dimension of $M_V$ to equal $0$,
$$
\ell(\mu) = 1 - g.
$$
We showed in Lemma \ref{lem-gammaminuspoint} that unmarked vertices are rational, and therefore   $\ell(\mu) = 1$.

If $v$ is adjacent to the $i$-th marked leg, recall that $\val (v) = k_i+3 -g$. Therefore, for the virtual dimension of $M_v$ to be $k_i+2$ it must be that $\ell(\phi) = 0$.

\end{remark}

\begin{proposition}
The virtual multiplicity $m_{(\Gamma,f)}$ can be written in terms of the local vertex multiplicities of $(\Gamma,f)$. Specifically, 
\[
m_{(\Gamma,f)} = \frac{1}{|\mathsf{Aut}(f)|} \cdot \prod_{e:\mathsf{C.E.}} w(e)\cdot \prod_V \mult_V(\Gamma,f).
\]
where $\mathsf{C.E.}$ stands for compact edge. 
\end{proposition}

\begin{proof}
The proof is a standard application of the degeneration formula~\cite{Che14,KLR}, and analogous statements can be found in~\cite{MR16} and Theorem~\ref{thm-flooralg} from earlier in this paper. As noted above, the orientation chosen is the only way in which to obtain a nonzero invariant, since the virtual dimension at each vertex must be zero to obtain a nonzero contribution in the degeneration formula. The product of edge expansion factors comes directly from the statement of the formula, while the automorphism factor arises from passing to a rigidified moduli space where the maps can be uniquely decomposed into their constituent components. We leave the bookkeeping to the reader.
\end{proof}

\begin{remark}\label{rem-factorsoftropmultareas}
Assume a floor-decomposed rigid tropical stable map $(\Gamma,f)$ contributes with non-zero multiplicity to a tropical descendant invariant $\langle ({\underline{\phi}}^-,{\underline{\mu}}^-)| \tau_{k_1}(p_1)\ldots\tau_{k_n}(p_n)|({\underline{\phi}}^+,{\underline{\mu}}^+)\rangle_{g}^{\trop}$ as in Definition \ref{def-tropdescGWI}. 
Let $V_1,\ldots, V_\ell$ be the vertices that do not support a marked point and are not adjacent to multiple edges forming flat cycles. By Lemma \ref{lem-3valent}, these vertices are trivalent. Then the factor 
\[
\prod_{e:\mathsf{C.E}} w(e)\cdot \prod_{i=1}^\ell \mult_{V_i}(\Gamma,f)
\]
appearing in the multiplicity $m_{(\Gamma,f)}$
is equal to the product of all (normalized) areas of triangles dual to the trivalent non-marked vertices in the dual subdivision, divided by the weights of fixed ends. Indeed, the Gromov--Witten invariant at trivalent vertices is $1$, and the gluing factors and edge weights together contribute the product of areas above, as follows from~\cite{GM053, MR08, Mar06}. 
\end{remark}

\subsection{Logarithmic floor multiplicity}

We keep the definition of a floor diagram used in the previous section but we change the multiplicity from a relative invariant to the corresponding logarithmic invariant. 

\begin{warning}
In what follows, to avoid overburdening the notation, we repurpose the symbols from the previous section, replacing the relative multiplicities with their logarithmic multiplicities.
\end{warning}

\begin{definition}\label{def-floormult-log}
Given  a floor diagram for $\mathbb{F}_k$, let $V$ be a vertex of genus $g_V$, size $s_V$ and with $\psi$-power $k_V$. Let $({\underline{\phi}}_V,{\underline{\mu}}_V)$ denote the expansion factors of the flags adjacent to $V$; the first sequence encodes the normal half edges, the second the thickened ones. We define the \textbf{logarithmic multiplicity $\mult(V)$} of $V$ to be the one-point stationary logarithmic descendant invariant 
$$\mult(V)=\langle ({\underline{\phi}}^-_V,{\underline{\mu}}^-_V)|\tau_{k_V}(pt)|({\underline{\phi}}^+_V,{\underline{\mu}}^+_V)\rangle^{\mathsf{log}}_{g_V}. $$
\end{definition}

\begin{definition}[Floor multiplicity for logarithmic geometries]\label{def-floorcount-log}
Fix discrete data as in Notation \ref{not-delta}.
We  define: $$\langle ({\underline{\phi}}^-,{\underline{\mu}}^-)|\tau_{k_1}(pt)\ldots\tau_{k_n}(pt)|({\underline{\phi}}^+,{\underline{\mu}}^+)\rangle_{g}^{\floor}$$
to be the weighted count of floor diagrams $D$ for $\mathbb{F}_k$ of degree $({\underline{\phi}},{\underline{\mu}})$ and genus $g$, with $n$ vertices with $\psi$-powers $k_1,\ldots,k_n$, such that $a$ equals the sum of all sizes of vertices, $a=\sum_{V=1}^n s_V$. 

Each floor diagram is counted with multiplicity $$\mult(D)=
\prod_{e \in C.E.} w(e)\cdot \prod_V \mult(V),$$
 where the second product is over the set $C.E.$ of  compact edges and $w(e)$ denotes their expansion factors; the third product ranges over all vertices $V$ and $\mult(V)$ denotes their multiplicities as in  Definition \ref{def-floormult-log}.
\end{definition}

\begin{theorem}\label{thm-floortrop}

Fixing all discrete invariants as in Notation \ref{not-delta}, the weighted count of floor diagrams equals the tropical descendant log Gromov--Witten invariant, i.e.\ we have
\begin{equation} \langle ({\underline{\phi}}^-,{\underline{\mu}}^-)|\tau_{k_1}(pt)\ldots\tau_{k_n}(pt)|({\underline{\phi}}^+,{\underline{\mu}}^+)\rangle_{g}^{\floor}= \langle ({\underline{\phi}}^-,{\underline{\mu}}^-)|\tau_{k_1}(pt)\ldots\tau_{k_n}(pt)|({\underline{\phi}}^+,{\underline{\mu}}^+)\rangle_{g}^{\trop}.
\end{equation}
\end{theorem}

\begin{proof}
The proof of this theorem is in two parts. Construction \ref{const-floor}  associates a floor diagram to a floor decomposed tropical stable map contributing to $\langle ({\underline{\phi}}^-,{\underline{\mu}}^-)|\tau_{k_1}(pt)\ldots\tau_{k_n}(pt)|({\underline{\phi}}^+,{\underline{\mu}}^+)\rangle_{g}^{\trop}$. 
By Proposition \ref{prop-floormult}, the weighted number of tropical stable maps yielding the same floor diagram $D$ under this procedure equals the multiplicity  $\mult(D)$ from Definition \ref{def-floorcount-log}.
\end{proof}

\begin{remark}
The key step in decomposing the tropical curve multiplicity into vertex terms was the transition from the normal crossings degeneration $\mathscr X$ to the non-proper double point degeneration $\mathscr X^\circ$, which required the vanishing results in this section. Once this has been done, the calculation can be undertaken in various formalisms. One could use the degeneration formula for relative geometries and expanded degenerations, see~\cite{Che14,Kim08, Li02}. Instead, one could avoid the expanded formalism by appealing to~\cite{KLR}. The transition from a normal crossings degeneration to a double point degeneration avoids the complexities of the logarithmic degeneration formula~\cite{R19}. 
\end{remark}

\subsection{Constructing floor diagrams from tropical curves} Let $(\Gamma,f)$ be a rigid floor decomposed tropical stable map contributing to 
\[
\langle ({\underline{\phi}}^-,{\underline{\mu}}^-)|\tau_{k_1}(pt)\ldots\tau_{k_n}(pt)|({\underline{\phi}}^+,{\underline{\mu}}^+)\rangle_{g}^{\trop}.
\] 

Because of the horizontally stretched point conditions, each marked point is either on a horizontal edge (resp.\ \textbf{elevator edge}) of $f(\Gamma)$ (i.e.\ an edge of primitive direction $(1,0)$) or on a part dual to a slice in the Newton subdivision. On each part dual to a slice, there is exactly one marked point.
Consider the preimage in $\Gamma$ under $f$ of a part dual to a slice, this is a subgraph that we call $\Gamma'$. Assume the slice in the Newton polygon has width $s>0$ (i.e.\ in plane coordinates, it is a slice between the lines $\{x=i\}$ and $\{x=i+s\}$ for some $i$). Since the image of $\Gamma'$ is fixed by exactly one point (and conditions on the coordinates of its horizontal edges), $\Gamma'$ consists of only rational connected components.  Furthermore, all but one of these components is just one edge which is mapped horizontally. This connected component (which contains $s$ ends of direction $(0,-1)$ and $s$ ends of direction $(k,1)$) is called a \textbf{floor of size $s$}. We refer to other connected components as horizontal edges passing through the floor. For an example, see Figures \ref{fig-floors1} and \ref{fig-floors2}.

\begin{construction}\label{const-floor}
Let $(\Gamma,f)$ be a non-superabundant floor decomposed tropical stable map contributing to $\langle ({\underline{\phi}}^-,{\underline{\mu}}^-)|\tau_{k_1}(pt)\ldots\tau_{k_n}(pt)|({\underline{\phi}}^+,{\underline{\mu}}^+)\rangle_{g}^{\trop}$. 

We associate a floor diagram $D$ contributing to $\langle ({\underline{\phi}}^-,{\underline{\mu}}^-)|\tau_{k_1}(pt)\ldots\tau_{k_n}(pt)|({\underline{\phi}}^+,{\underline{\mu}}^+)\rangle_{g}^{\floor}$ to $(\Gamma,f)$ by contracting each floor to a vertex; also marked points adjacent to only horizontal edges are considered vertices. The vertices are equipped with:
\begin{itemize}
\item the $\psi$-power $k_i$ of the adjacent marked point $i$, 
\item the size $s_i$ (i.e. the width) of the  dual slice of the Newton polygon for vertices corresponding to a floor; $s_i  = 0$ for marked points on elevators,
\item the genus $g_i$  of the vertex adjacent to the marked end $i$ in the tropical curve.
\end{itemize}
We thicken flags if they come from half-edges of $f(\Gamma)$ which are adjacent to a marked point.
\end{construction}
\begin{proof} We  show that Construction \ref{const-floor} yields a floor diagram of the right degree and genus.
Because of the horizontally stretched point conditions, we obtain a graph $D$ on a linearly ordered vertex set. 

The balancing condition satisfied by $(\Gamma,f)$ implies that the signed sum of expansion factors of edges adjacent to vertex $i$ of the floor diagram equals $-k\cdot s_i$. 

By Lemma \ref{lem-gammaminuspoint}, removing from the subgraph underlying a floor of size $s_i$ the marked end $i$ together with its end vertex yields  connected components each containing at most one of the $2s_i$ ends of direction $(0,-1)$ resp.\ $(k,1)$.
It follows that the valence of the vertex adjacent to the $i$-th mark is $2s_i$ plus one (for the marked end itself) plus the number of adjacent horizontal edges. The latter correspond to the thick flags in the floor diagram $D$. Thus at vertex $i$ of $D$, $(k_i-g_i+3)-1-2s_i$ edges are thickened, as required. Furthermore, each horizontal edge of $\Gamma$ must be fixed, either by a condition on the $y$-coordinates of ends, or by a marked point. It cannot be fixed more than once because of the genericity of the conditions. It follows that every edge of the associated floor diagram $D$ has precisely one thickened flag, as required.
Since all floors of $(\Gamma,f)$ are rational, the genus of $D$ is $g$. Obviously, the degree of $D$ is $({\underline{\phi}},{\underline{\mu}})$.
Thus $D$ is a floor diagram contributing to 
$\langle ({\underline{\phi}}^-,{\underline{\mu}}^-)|\tau_{k_1}(pt)\ldots\tau_{k_n}(pt)|({\underline{\phi}}^+,{\underline{\mu}}^+)\rangle_{g}^{\floor}$.

\end{proof}

\begin{proposition}\label{prop-floormult}
Let $D$ be a floor diagram contributing to 
\[
\langle ({\underline{\phi}}^-,{\underline{\mu}}^-)|\tau_{k_1}(pt)\ldots\tau_{k_n}(pt)|({\underline{\phi}}^+,{\underline{\mu}}^+)\rangle_{g}^{\floor}.
\] 
The weighted number of tropical stable maps contributing to 
\[
\langle ({\underline{\phi}}^-,{\underline{\mu}}^-)|\tau_{k_1}(pt)\ldots\tau_{k_n}(pt)|({\underline{\phi}}^+,{\underline{\mu}}^+)\rangle_{g}^{\trop}
\] 
that yield $D$ under the procedure described in Construction \ref{const-floor} equals  $\mult(D)$.
\end{proposition}
{
\begin{proof}


Let $(\Gamma,f)$ be a tropical stable map that yields $D$ using Construction \ref{const-floor}.


Then $(\Gamma,f)$ is rigid and all its non-marked vertices which are not adjacent to flat cycles are trivalent by the above. Using our convention of marking horizontal ends, it follows also that $(\Gamma,f)$ has no nontrivial automorphisms besides the ones coming from multiple edges in flat cycles.

Following Definition~\ref{def: local-vertex-multiplicity} and using Remark \ref{rem-factorsoftropmultareas}, $(\Gamma,f)$ contributes a product of
\begin{enumerate}
\item areas of triangles dual to non-marked vertices and factors $\frac{1}{w}$ for the weights of fixed ends, 
\item local vertex multiplicities for non-marked vertices adjacent to multiple edges in flat cycles, and
\item local marked vertex multiplicities $\mult_V(\Gamma,f)$.
\end{enumerate}
Every compact edge $e$ of $D$ of weight $w(e)$ comes from a bounded edge $e'$ of $\Gamma$ of weight $w(e)$. Since $e$ has precisely one non-thickened flag, $e'$ is adjacent to precisely one trivalent vertex $V$ not adjacent to a marked point or a flat cycle (see Lemma \ref{lem-3valent}).  Denote by $e''$ an edge in the floor which is adjacent to $V$.
Every non-horizontal edge in a floor is of direction $(0,1)+c\cdot (1,0)$ for some $c$ (by the balancing condition, the fact that the floor contains no cycles, and since we can connect every edge to an end of direction $(0,-1)$), and so the area of the triangle dual to $V$ (formed by the duals of $e'$ and  $e''$) is $w(e)$. 

Consider now the case of $m$ multiple edges, forming $m-1$ flat cycles as in Proposition \ref{lem-cyclesvisible} (C). One of the adjacent vertices carries a marked point. Thus for the other, say $V$, the multiple edges locally yield fixed ends. We can vary the position of the locally fixed ends without changing the local multiplicity. Thus, $\mult_V(\Gamma,f)$ decomposes as $m$ factors, each corresponding to a trivalent vertex with an adjacent horizontal fixed end of weight $w(e)$ where $e$ is the corresponding edge of $D$. Using the same arguments as before, $V$ contributes a product of the weights of its adjacent horizontal edges.

A non-fixed end of $\Gamma$ has to be adjacent to a marked point by rigidity, so it is not adjacent to a trivalent vertex as above. 
A fixed end of $\Gamma$ is adjacent to a trivalent vertex whose dual triangle has area $w(e)$ by the above. 

Altogether we can see that the first two items above --- the product over all areas of triangles dual to non-marked vertices in the dual subdivision of $(\Gamma,f)$ times the multiplicities of non-marked vertices adjacent to multiple edges forming flat cycles divided by factors $w$ for fixed ends --- equals the product of weights of the compact edges of $D$.

We cut $(\Gamma,f)$ into floors. Each floor $(\Gamma',f')$ can be viewed as a tropical stable map contributing to the count $$\langle ({\underline{\phi}}^-_V,{\underline{\mu}}^-_V)|\tau_{k_V}(pt)|({\underline{\phi}}^+_V,{\underline{\mu}}^+_V)\rangle_{g_V} $$ which gives the multiplicity of the floor viewed as a vertex $V$ of $D$.
As such, the floor contributes its tropical multiplicity, which is again a product as above. 

Let $v$ be a vertex of $D$. By Theorem \ref{thm-corres}, $\mult(v)$ equals the weighted sum of all floors $(\Gamma',f')$ of some $(\Gamma,f)$ that map to $v$ under Construction \ref{const-floor}. 
In this weighted count, each summand contributes with its tropical multiplicity as above. Since every end of $\Gamma'$ which is not adjacent to the marked point in $\Gamma'$ has to be fixed by rigidity, the only contribution we have for the whole floor is the local vertex multiplicity $\mult_V(\Gamma',f')$ of the vertex $V$ of $\Gamma'$ adjacent to the marked point. Thus, $\mult(v)$ equals the weighted sum over all floors that can possibly be inserted, each counted with the factor $\mult_V(\Gamma',f')$ where $V$ is the vertex adjacent to the marked point.



Since we can freely combine floors by gluing them to elevator edges as imposed by $D$, $\mult(D)$ equals the weighted count of all tropical stable maps contributing to the invariant \[
\langle ({\underline{\phi}}^-,{\underline{\mu}}^-)|\tau_{k_1}(pt)\ldots\tau_{k_n}(pt)|({\underline{\phi}}^+,{\underline{\mu}}^+)\rangle_{g}^{\trop}
\] 
and yielding $D$ under the procedure described in Construction \ref{const-floor}, where each tropical stable map is counted with a product of weights for the compact edges of $D$ times $\mult_V(\Gamma',f')$ where $V$ is the vertex adjacent to the marked point. 
We have seen above that the product of weights for the compact edges of $D$ equals the product of the areas of triangles dual to non-marked edges times the weights of multiple edges, divided by the weights of the fixed ends. Thus $\mult(D)$ equals the weighted count of all tropical stable maps yielding $D$, each weighted with its tropical multiplicity. The statement follows.

\end{proof}}


\section{Floor diagrams via the operator theory}\label{sec-fock}

In this section we build on work of Cooper and Pandharipande~\cite{CP12} and Block and G\"ottsche~\cite{BG14} and express relative descendant Gromov--Witten invariants of Hirzebruch surfaces as matrix elements for an operator on a Fock space. The results of this section continue to hold if one replaces the stationary logarithmic descendants of the previous section. We begin the section by reviewing the formalism of Fock spaces in our context.

Let $\mathcal{H}$ denote the algebra presented with generators $a_n,b_n$ for $n\in \mathbb{Z}$ satisfying the commutator relations
\begin{equation}
 [a_n,a_m]=0, \;\;\; [b_n,b_m]=0, \;\;\; [a_n, b_m]=n\cdot \delta_{n,-m},
 \end{equation}
 where $\delta_{n,-m}$ is the Kronecker symbol. We let $a_0=b_0=0$.

The Fock space $F$ is the vector space generated by letting the generators $a_n, b_n$ for $n<0$ act freely (as linear operators) on the so-called vacuum vector $v_\emptyset$. We define $a_n\cdot v_\emptyset=b_n\cdot v_\emptyset=0$ for $n>0$. For a pair of partitions ${\underline{\phi}}=(\varphi_1,\ldots,\varphi_{n_1})$ and ${\underline{\mu}}=(\mu_1,\ldots,\mu_{n_2})$, we denote
\begin{equation}
v_{{\underline{\phi}},{\underline{\mu}}}=\frac{1}{|\Aut({\underline{\phi}})|\cdot |\Aut({\underline{\mu}})|}a_{-\varphi_1}\cdot \ldots \cdot a_{-\varphi_{n_1}} \cdot b_{-\mu_1}\cdot \ldots \cdot b_{-\mu_{n_2}}\cdot v_{\emptyset}.
\end{equation}
 The vectors $\{v_{{\underline{\phi}},{\underline{\mu}}}\}$ indexed by pairs of partitions ${\underline{\phi}}$, ${\underline{\mu}}$ form a basis for $F$. We define an inner product on $F$ by declaring $\langle v_\emptyset | v_\emptyset \rangle=1$ and $a_n$ to be the adjoint of $a_{-n}$, $b_n$ of $b_{-n}$.  The structure constants for the inner product in the two-partition basis are:
  \begin{equation} \langle v_{{\underline{\phi}},{\underline{\mu}}}| v_{{\underline{\phi}}',{\underline{\mu}}'}\rangle = \prod \varphi_i\cdot \prod \mu_i \cdot \frac{1}{|\Aut({\underline{\phi}})|}\cdot \frac{1}{|\Aut({\underline{\mu}})|}\cdot \delta_{{\underline{\phi}}, {\underline{\mu}}'}\cdot \delta_{{\underline{\mu}},{\underline{\phi}}'}. \end{equation}

Following standard conventions, for $\alpha, \beta \in F$ and an operator $A\in \mathcal{H}$, we write $\langle \alpha|A|\beta\rangle$ for $\langle \alpha|A\beta\rangle$. Such expressions are referred to as \textbf{matrix elements}. We write $\langle A \rangle$ for $\langle v_\emptyset |A|v_\emptyset \rangle$; such a value is called a \textbf{vacuum expectation}.

We also introduce \textbf{normal ordering} of operators in $\mathcal{H}$.  If $c_i, i=1,\dots, n$ are operators in $\mathcal{H}$, then the normally ordered product $:\prod_{i=1}^n c_i:$ reorders the $c_i$ so that any $c_i$ with $i>0$ occurs after the $c_j$ with $j<0$.  For example, we have $$:a_2b_{-2}a_2a_{-1}\!: \ \ = \ \ b_{-2}a_{-1}a_2a_2.$$

As before, we fix $k\in \NN$ to identify a Hirzebruch surface $\mathbb{F}_k$.

\begin{definition}\label{def-operator}
Let $m\in \NN_{>0}$, $l$, $s$ and $g\in \NN$ be given. Let ${\bf{z}}\in (\ZZ\smallsetminus \{0\})^m$ satisfy $\sum_{i=1}^m z_i=-k\cdot s$. 
Denote ${\underline{\mu}}=(z_1,\ldots,z_{l+2-2s-g})$ and ${\underline{\phi}}=(z_{l+2-2s-g+1}\ldots,z_m)$, and let superscripts $\pm$ denote the subsets of positive (resp. negative) entries.  

For a formal variable $u$, define $$\hat{a}_n=\begin{cases}u a_n &\mbox{ if } n<0\\ a_n &\mbox{ if } n>0 \end{cases} \;\;\;\; \mbox{ and } \;\;\;\;\hat{b}_n=\begin{cases}u b_n &\mbox{ if } n<0\\ b_n &\mbox{ if } n>0 \end{cases}.$$

We define the following series of operators in $\mathcal{H}[t,u]$, indexed by  $l\in \NN$:
\begin{multline*}M_l=\sum_{g\in \NN} u^{g-1} \sum_{s\in \NN} t^s \sum_{m\in \NN_{>0}} \sum_{\bf{z}\in \ZZ^{m}} 
\left\langle {(\underline{\phi}}^-,{\underline{\mu}}^- )| 
\tau_{l}(pt) 
|({\underline{\phi}}^+,{\underline{\mu}}^+)\right\rangle^{\mathsf{rel}}_{g} \cdot \\:\hat{b}_{z_1}\cdot \ldots \cdot \hat{b}_{z_{l+2-2s-g}}\cdot \hat{a}_{z_{l+2-2s-g+1}}\cdot \ldots \cdot \hat{a}_{z_m}:
\end{multline*}
where the fourth sum is taken over all $\bf{z}$ satisfying $\sum_i z_i=-k\cdot s$ (where $s$ is the index of the second sum), and where the 
 one-point Gromov--Witten invariant $\langle {(\underline{\phi}}^-,{\underline{\mu}}^- )| \tau_{l}(pt) |({\underline{\phi}}^+,{\underline{\mu}}^+)\rangle^{\mathsf{rel}}_{g}$
depends on the indices $l$, $g$ and $\bf{z}$ as above.


\end{definition}



\begin{remark}
Consider the operator $M_0$. It has only two summands for $s$, $s=0$ and $s=1$, since $2-2s-g<0$ for $s>1$. 
If $s=0$, the curve class in the Gromov--Witten invariant $\langle ({\underline{\phi}}^-,{\underline{\mu}}^-)|\tau_{0}(pt)| ({\underline{\phi}}^+,{\underline{\mu}}^+)\rangle^{\mathsf{rel}}_{g}$ is a multiple of the class of a fiber. This implies that the moduli space of maps is non-empty only if $g=0$ and $m=2$.
The invariant $\langle {{\underline{\mu}}}^-|\tau_{0}(pt)|{{\underline{\mu}}}^+\rangle^{\mathsf{rel}}_{0}$, for ${\underline{\mu}} = (d,-d)$ is readily seen to be $1$: there is a unique map of degree $d$ from a rational curve to the fiber identified by the point condition, fully ramified at $0$ and $\infty$ (the intersections of the sections with the given fiber). Such a map has no automorphisms because we have marked one point on the rational curve.

If $s=1$, we must have $g=0$ and no $b$ factors. The invariants  $\langle {{\underline{\phi}}}^-| \tau_{0}(pt)| {{\underline{\phi}}}^+\rangle^{\mathsf{rel}}_{0}$ are all $1$ by the genus $0$ correspondence theorem and a tropical computation, see~\cite{NS06,R15b}.


So we have
$$M_0=\sum_{z_1+z_2=0}b_{z_1}\cdot b_{z_2}+ \sum_{{\underline{\phi}}\in (\ZZ\smallsetminus \{0\})^m} t\cdot u^{\#\{{\underline{\phi}}^-\}-1} a_{z_1}\cdot \ldots\cdot a_{z_m},  $$
where the second sum goes over all ${\mathbf{z}}\in (\ZZ\smallsetminus \{0\})^m$ satisfying $\sum z_i=-k$.  Here the normal ordering is unnecessary since the $a_i$ commute amongst themselves, as do the $b_j$.  
Since the genus can be computed from the Euler characteristic of the underlying Feynman graphs, the variable $u$ is superficial in this scenario. Setting $u=1$, we obtain the operator $H_k(t)$ defined in \cite{BG14}, Theorem 1.1. Our family of operators $M_l$ generalizes the operator of Block-G\"ottsche to one operator for each power of descendant insertions.
\end{remark}

\begin{theorem}\label{thm-operator}

With discrete data fixed 
 as in Notation \ref{not-delta}, 
the disconnected relative descendant Gromov--Witten invariant $\langle ({\underline{\phi}}^-,{\underline{\mu}}^-)|\tau_{k_1}(pt)\ldots\tau_{k_n}(pt)|({\underline{\phi}}^+,{\underline{\mu}}^+)\rangle_{g}^{\mathsf{rel},\bullet}$ equals the matrix element
\begin{align}\label{eq-operatornonmarked}
\begin{split}
&\langle ({\underline{\phi}}^-,{\underline{\mu}}^-)|\tau_{k_1}(pt)\ldots\tau_{k_n}(pt)|({\underline{\phi}}^+,{\underline{\mu}}^+)\rangle_{g}^{\mathsf{rel},\bullet}\\
=& \frac{|\Aut({\underline{\mu}})|}{\prod |\mu_i|}\frac{|\Aut({\underline{\Phi}})|}{\prod |\phi_i|}\left\langle v_{{\underline{\mu}}^-,{\underline{\phi}}^-}\left|\mbox{ Coeff}_{t^a u^{g+\ell({\underline{\phi}}^-) +\ell({\underline{\mu}}^-)-1}}\Big(\prod_{i=1}^n M_{k_i} \Big) \right|v_{{\underline{\mu}}^+,{\underline{\phi}}^+}\right\rangle,
\end{split}
\end{align}
where the operators $M_{k_i}$ are as defined in Definition \ref{def-operator}, and  for a series of operators $M\in \mathcal{H}[t,u]$ $\mbox{Coeff}_{t^a u^h}(M)\in \mathcal{H}$ denotes the $t^au^h$-coefficient.
\end{theorem}

\noindent{\bf Important detail.} Notice the order of the partitions is switched on the two sides of Equation \eqref{eq-operatornonmarked}, thus the $\mu_i$ entries are associated to $a$ variables and vice-versa.

Before we start a formal proof of Theorem \ref{thm-operator}, we make a relevant definition and recall an important tool for the proof.


After translating the matrix element in Equation \eqref{eq-operatornonmarked} to a vacuum expectation, we  compute it as the weighted sum over \textbf{Feynman graphs} associated to each monomial contributing to the expectation. This can be viewed as a variant of Wick's theorem \cite{Wic50} and is proved in Proposition 5.2 of \cite{BG14}.
Generalizing the situation in \cite{BG14}, the Feynman graphs in question are essentially floor diagrams and Theorem \ref{thm-operator} follows because of a natural weighted bijection of Feynman graphs and floor diagrams.

\begin{definition}\label{def-Feynmangraph}
Let $P=m_+\cdot m_1\cdot \ldots \cdot m_n\cdot m_{-}$ be  a product of  monomials in the variables $a_s$ or $b_s$, such that:
\begin{itemize}
\item for each monomial,  all operators with negative indices stand left of all operators with positive indices;
\item $m_+$ contains only positive factors  (with $s>0$);
\item $m_-$ contains only negative factors (with $s<0$).
\end{itemize}
We associate graphs to $P$ called \textbf{Feynman graphs} for $P$, via the following algorithm.

\noindent{\bf Step 1: local pieces.}
To any monomial $m_i$, associate a star graph with vertex denoted $v_i$: for each factor $a_s$ appearing in $m_i$, draw a (non-thickened) edge germ of weight $|s|$ which is directed to the left if $s<0$ and to the right if $s>0$. For each factor $b_s$,  draw a thickened edge germ of weight $|s|$ which is directed to the left if $s<0$ and to the right if $s>0$.

To the special monomials $m_+, m_-$ associate a collection of disconnected, marked edge germs of weight equal to the absolute value of the index of each operator appearing in the monomials. Thicken the  germs corresponding to the operators $b_s$.

\noindent{\bf Step 2: Feynman fragment.}
We call the Feynman fragment associated to $P$ the disconnected graph obtained by  linearly ordering the union of all the local pieces: first come the edge germs relative to $m_+$, then  vertices $v_i$ (ordered according to their index $i$, and finally the edge germs corresponding to $m_-$.

\noindent{\bf Step 3: filling the gaps.} A Feynman graph completing the Feynman fragment is any (marked, weighted, ordered) graph obtained by promoting edge germs to half edges, and gluing pairs of half edges until there is none left. A pair of half edges may be glued if:
\begin{itemize}
\item one is directed to the right and the other to the left, and the vertex adjacent to the germ directed to the right is smaller than the one adjacent to the germ directed to the left,
\item the two edge germs have the same weight, and
\item one edge germ is thickened and one is not.
\end{itemize}
\end{definition}

\begin{example}\label{ex-P}
Let $P$ be the product 
$$P=(b_2\cdot a_1\cdot a_2) \cdot (b_{-2}\cdot b_2) \cdot (a_{-2}\cdot b_{-1}\cdot a_2)\cdot (b_{-2}\cdot a_{-2}\cdot a_1\cdot a_1)\cdot (b_{-1}\cdot b_1)\cdot (b_{-1}\cdot a_{-1}),$$ where the factors $m_i$ are separated by parentheses.
Following Definition \ref{def-Feynmangraph}, a Feynman graph for $P$ is any graph completing the Feynman fragment depicted in Figure \ref{fig-edgegerms}.
\begin{figure}
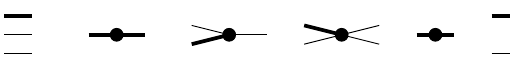
\caption{The weighted, directed, possibly thickened edge germs corresponding to the product $P$ in Example \ref{ex-P}. (We drop the marking of edge germs in the picture.)}
\label{fig-edgegerms}
\end{figure}
In Figure \ref{fig-exFeynman}, the dotted lines  suggest a  way to complete the  fragment to a Feynman graph for $P$. After removing all external half edges, we recognize the floor diagram depicted in Figure \ref{fig-floordiagram2a}. 
\begin{figure}
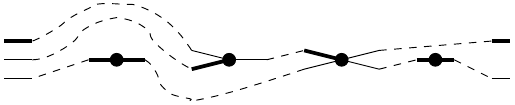
\caption{A Feynman graph completing the edge germs associated to the product $P$ in Example \ref{ex-P}.}\label{fig-exFeynman}
\end{figure}
\end{example}

\begin{proposition}[Wick's Theorem, see Proposition 5.2 of \cite{BG14}]\label{prop-wick}
The vacuum expectation $\langle P\rangle$ for a product $P$ as in Definition \ref{def-Feynmangraph} equals the weighted sum of all Feynman graphs for $P$, where each Feynman graph is weighted by the product of weights of all edges (interior edges and ends).
\end{proposition}

A detailed proof of this proposition may be found in \cite{BG14}. Here we provide an intuitive and informal description of the mechanism that underlies the proof, as we feel this will be more beneficial to a reader who is not already an expert on these techniques.

\begin{proof}
In the product $P$, we take the right most factor $a_i$ or $b_i$ with $i>0$, and try to move it to the right. To simplify notations, let us assume that this right most factor is $a_i$ for some $i>0$. If this factor $a_i$ reaches the very right in a contribution we produce in this way (i.e.\ ends up being the right most factor of a contributing term), then we obtain zero since by definition $a_i\cdot v_\emptyset=b_i\cdot v_\emptyset=0$ for $i>0$. The commutator relations produce several contributing terms for $\langle P \rangle$ when moving $a_i$ to the right. We can make $a_i$ jump over any $a_j$, or $b_k$ with $k\neq -i$. If $a_i$ is the left neighbour of $b_{-i}$ however, the commutator relation replaces $a_ib_{-i}$ by $b_{-i}a_{i}+i$. That is, we get two summands, one in which we manage to move $a_i$ further to the right, and one where we cancel this factor together with its neighbour $b_{-i}$. 

With both summands, we continue moving the right most factor with positive index right. For the summand in which we cancel $a_i$ together with a factor of $b_{-i}$ appearing right of $a_i$ in $P$, we add to the Feynman fragment of $P$ by drawing an edge connecting the germ corresponding to $a_i$ and the germ corresponding to $b_{-i}$.

By following this procedure we draw all Feynman graphs completing the Feynman fragment for $P$. Each Feynman graph corresponds to a way to group the factors of $P$ in pairs $\{a_i,b_{-i}\}$ corresponding to edges completing the corresponding marked edge germs. Each such pair produces a contribution of $i$ because of the commutator relations, so altogether each Feynman graph should be counted with  weight  equal to the product of its edge weights to produce $\langle P \rangle$.
\end{proof}

\begin{proof}[Proof of Theorem \ref{thm-operator}:]
First we express the matrix element in Equation \eqref{eq-operatornonmarked} as a vacuum expectation:
\begin{align} \label{eq-expval}
&\frac{|\Aut({\underline{\mu}})|}{\prod |\mu_i|}\frac{|\Aut({\underline{\phi}})|}{\prod |\varphi_i|}\left\langle v_{{\underline{\mu}}^-,{\underline{\phi}}^-}\left|M \right|v_{{\underline{\mu}}^+,{\underline{\phi}}^+}\right\rangle = \nonumber\\ 
& \frac{|\Aut({\underline{\mu}})|}{\prod |\mu_i|}\frac{|\Aut({\underline{\phi}})|}{\prod |\varphi_i|}
\frac{1}{|\Aut({\underline{\phi}}^+)| |\Aut({\underline{\mu}}^+)|}\frac{1}{|\Aut({\underline{\phi}}^-)| |\Aut({\underline{\mu}}^-)|}\cdot
 \nonumber\\
 &
\left\langle v_{\emptyset}\left| \prod_{\mu_i \in \underline{\mu}^-}a_{|\mu_i|} \prod_{\varphi_i \in \underline{\phi}^-} b_{|\varphi_i|} \ \
M \prod_{\mu_i \in \underline{\mu}^+}a_{-\mu_i} \prod_{\varphi_i \in \underline{\phi}^+} b_{-\varphi_i} \right|v_{\emptyset}\right\rangle=\nonumber\\
&
\frac{1}{\prod |\varphi_i|\prod |\mu_i|}\left\langle \prod_{\mu_i \in \underline{\mu}^-}a_{|\mu_i|} \prod_{\varphi_i \in \underline{\phi}^-} b_{|\varphi_i|}\ \
M \prod_{\mu_i \in \underline{\mu}^+}a_{-\mu_i} \prod_{\varphi_i \in \underline{\phi}^+} b_{-\varphi_i} \right\rangle
\end{align}

 By Theorem \ref{thm-flooralg}, the left-hand side  in Equation \eqref{eq-expval} equals an appropriate count of floor diagrams. By Proposition \ref{prop-wick}, each term contributing to the right-hand side can be expressed in terms of a weighted count of suitable Feynman diagrams. We show that the floor diagrams contributing to the left-hand side are essentially equal to the Feynman graphs contributing to the right, and that they are counted with the same weight on both sides.

Expand the left-hand side so that it becomes a sum of vacuum expectations, where each summand is of the form $w_P \cdot P$ such that $w_P$ is a number and $P=m_+\cdot \ldots\cdot m_{-}$ a monomial as described in Definition \ref{def-Feynmangraph}. For each summand, 
$$m_+=\prod_{\mu_i \in \underline{\mu}^-}a_{|\mu_i|}\cdot \prod_{\varphi_i \in \underline{\phi}^-} b_{|\varphi_i|} \ \ \mbox{ and } \ \ \ m_{-}=\prod_{\mu_i \in \underline{\mu}^+}a_{-\mu_i}\cdot \prod_{\varphi_i \in \underline{\phi}^+} b_{-\varphi_i}.$$
 A factor $m_i$ for $i=1,\ldots,n$ comes from a summand of $M_{k_i}$, i.e.\ is of the form 
$$\langle ({\underline{\phi}}^-,{\underline{\mu}}^-)|\tau_{k_i}(pt)|({\underline{\phi}}^+,{\underline{\mu}}^+)\rangle_{g_i}\cdot :\hat{b}_{z_1}\cdot \ldots \cdot \hat{b}_{z_{k_i+2-2s_i-g}}\cdot \hat{a}_{z_{k_i+2-2s_i-g_i+1}}\cdot \ldots \cdot \hat{a}_{z_m}:,$$ 
where $s_i$ is encoded in the power of $t$ and $g_i$ in the power of $u$.

  Enrich the Feynman fragment for $P$ by adding three numbers to each vertex $i$, namely the $\psi$-power $k_i$ (imposed by the operator $M_{k_i}$ of which the factor corresponding to vertex $i$ is taken), the size $s_i$ (imposed by the power of $t$) and the genus $g_i$ (imposed by the power of $u$). Any Feynman diagram completing this Feynman fragment is by definition a weighted loop-free graph with ends on the linearly ordered vertex set $v_1,\ldots,v_n$. After removing all external half edges, the conditions (1), (2) and (3) we impose in the definition of a floor diagram (Definition \ref{def-floor}) are satisfied.
By definition of the operator $M_l$ (see Definition \ref{def-operator}), the signed sum of weights of edges adjacent to a vertex equals $-k\cdot s_i$, so  condition (5) is satisfied. By definition of the operator $M_l$, in each factor $m_i$, exactly $k_i+2-2s_i-g_i$ factors are $b$-operators and thus correspond to thickened edge germs, so condition (4) is satisfied.  

Since we take the $t^a$ coefficient of the product $M_{k_1}\cdot \ldots \cdot M_{k_n}$ for the operator in Equation (\ref{eq-operatornonmarked}), we obtain floor diagrams satisfying $a=\sum s_i$. The degree $({\underline{\phi}},{\underline{\mu}})$ is determined by the boundary conditions. To see that the floor diagram is of the right genus, notice that the variable $u$ is in charge of genus. Let us build a Feynman graph from the left to the right, starting with the left ends, and adding in vertex after vertex from $1$ to $n$, taking the change in genus into account in each step. The genus of the graph consisting of $ \ell({\underline{\phi}}^-) +\ell({\underline{\mu}}^-)$ left ends (at first disconnected) has genus $-\ell({\underline{\phi}}^-) -\ell({\underline{\mu}}^-)+1$. For the vertex $i$ of local genus $g_i$, by definition of the operator $M_l$, we get a contribution of $u^{g_i-1}$, and we get as many additional factors of $u$ as the vertex has incoming edges (by the $\hat{a}_i$ resp.\ $\hat{b}_i$ convention). Since $h$ incoming edges potentially close up $h-1$ cycles, the vertex $i$ increases the genus by $g_i+h_i-1$, where $h_i$ denotes the number of incoming edges. By taking the $u^{g+\ell({\underline{\phi}}^-) +\ell({\underline{\mu}}^-)-1}$ coefficient in total, we thus obtain floor diagrams of genus $g$.

Each Feynman graph for $P$ can thus be viewed (after removing external half edges) as a floor diagram contributing to the left-right-handhand side, and vice versa, each floor diagram gives a Feynman graph.


It remains to show that a Feynman graph and the corresponding floor diagram contribute to Equation \eqref{eq-operatornonmarked}  with the same multiplicity. For the right-hand side, note that a Feynman graph contributes with the product of the weight of all of its egdes times  the coefficient $w_P$ of the product $P$ in the expansion of the product of the $M_l$-operators. Dividing by the factor $\frac{1}{|\mu_i|}\frac{1}{|\varphi_i|}$ (see the right-hand side of Equation \eqref{eq-expval}), we see that we are giving the Feynman graph weight equal to the product of the weights  of its internal edges times the factor $w_P = \prod_{v=1}^n \langle({\underline{\phi}}^-_v,{\underline{\mu}}^-_v)| \tau_{k_v}(pt)|({\underline{\phi}}^+_v,{\underline{\mu}}^+_v)\rangle^{\mathsf{rel}}_{g_v}$. This is precisely the weight of the corresponding floor diagram in Equation \eqref{def-floorcount}. 



 
\end{proof}

\bibliographystyle{siam}
\bibliography{PsiFockSurfaces}

\end{document}